\documentclass[11pt]{scrartcl}
\usepackage[T1]{fontenc}
\usepackage{lmodern}
\usepackage{amsmath,amsthm,amssymb,amsfonts}
\usepackage{enumerate}
\usepackage[dvipsnames,svgnames,table]{xcolor}
\usepackage{pgf,tikz}
\usetikzlibrary{calc}
\usetikzlibrary{backgrounds}
\usetikzlibrary{arrows.meta}
\usepackage[colorlinks=true,linkcolor=RoyalBlue,citecolor=PineGreen,urlcolor=RoyalBlue]{hyperref}
\usepackage{booktabs}
\usepackage[nameinlink]{cleveref}
\usepackage{enumitem}

\newtheorem{theorem}{Theorem}[section]
\newtheorem{proposition}[theorem]{Proposition}
\newtheorem{lemma}[theorem]{Lemma}
\newtheorem{corollary}[theorem]{Corollary}
\newtheorem{conjecture}[theorem]{Conjecture}

\theoremstyle{definition}

\newtheorem{example}[theorem]{Example}

\theoremstyle{remark}

\newcommand{\rank}{\operatorname{rank}}
\newcommand{\mlt}{\operatorname{mlt}}
\newcommand{\gcr}{\operatorname{gcr}}

\newcommand{\dbgraph}{$D\!B$}

\colorlet{colbg}{white}
\colorlet{colfg}{black}
\colorlet{colgraphv}{colfg!75!white}
\colorlet{colgraphe}{colfg!55!white}
\colorlet{colG}{DarkSeaGreen}
\definecolor{colR}{HTML}{CC6677}
\definecolor{colO}{HTML}{DDCC77}
\definecolor{colB}{HTML}{6699CC}
\colorlet{colY}{Gold!90!black}
\colorlet{colGl}{colG!50!white}
\colorlet{colOl}{colO!50!white}
\colorlet{colGrayl}{black!15!white}

\tikzstyle{vertex}=[fill=colgraphv,circle,inner sep=0pt, minimum size=4pt]
\tikzstyle{hvertex}=[vertex,minimum size=6pt,fill=colR]
\tikzstyle{edge}=[line width=1.5pt,colgraphe]
\tikzstyle{hedge}=[line width=2pt,colB]

\tikzstyle{labelsty}=[font=\scriptsize]

\title{Rigidity of nearly planar classes of graphs}
\author{%
    Sean Dewar\thanks{School of Mathematics, University of Bristol. E-mail: \texttt{sean.dewar@bristol.ac.uk}} \and 
    Georg Grasegger\thanks{Johann Radon Institute for Computational and Applied Mathematics (RICAM), Austrian Academy of Sciences. E-mail: \texttt{georg.grasegger@ricam.oeaw.ac.at}} \and 
    Eleftherios Kastis\thanks{School of Mathematical Sciences, Lancaster University. E-mail: \texttt{e.kastis@lancaster.ac.uk}} \and
    Anthony Nixon\thanks{School of Mathematical Sciences, Lancaster University. E-mail: \texttt{a.nixon@lancaster.ac.uk}} \and
    Brigitte Servatius\thanks{Mathematical Sciences, WPI. E-mail: \texttt{bservat@wpi.edu}}}

\begin{document}
\date{}
\maketitle

\begin{abstract}
We explore the rigidity of generic frameworks in 3-dimensions whose underlying graph is close to being planar. Specifically we consider apex graphs, edge-apex graphs and their variants and prove independence results in the generic 3-dimensional rigidity matroid adding to the short list of graph classes for which 3-dimensional rigidity is understood. We then analyse global rigidity for these graph classes and use our results to deduce bounds on the maximum likelihood threshold of graphs in these nearly planar classes.
\end{abstract}



\section{Introduction}

A bar-joint framework $(G,p)$ in $\mathbb{R}^d$ is an ordered pair consisting of a finite simple graph $G=(V,E)$ and a realisation $p:V\rightarrow \mathbb{R}^d$. Given such a structure analysing whether and how it can deform has a range of obvious practical applications. Mathematically, the problem is typically analysed by a linearisation known as infinitesimal rigidity \cite{AsimowRoth} which gives rise to the $d$-dimensional \emph{rigidity matroid} $\mathcal{R}_d(G,p)$. In this article we are principally interested in generic frameworks, for which the matroid depends only on the graph and we tend to drop the $G$ and refer to the $d$-dimensional \emph{generic rigidity matroid} as $\mathcal{R}_d$.

When $d=1$ then $\mathcal{R}_d$ is precisely the cycle matroid of the graph. The case when $d=2$ is also well understood due to Pollaczek-Geiringer \cite{Geiringer,Laman}. However, it is a long-standing problem in rigidity theory to understand the nature of the $d$-dimensional generic rigidity matroid when $d>2$ (see \cite{GGJN, GSS, JJ05, MR3838323} inter alia). For obvious reasons attention is focused on the case when $d=3$. Here, among the small number of special cases that are understood is the case of planar graphs.

\begin{theorem}[\cite{MR0400239}]\label{thm:gluck}
Every planar graph is $\mathcal{R}_3$-independent and a planar graph is $\mathcal{R}_3$-rigid if and only if it is a triangulation.
\end{theorem}

Limited progress has been made by considering graphs embeddable on other surfaces. Impressively Fogelsanger~\cite{Fogelsanger} proved that the graph of a triangulation of any surface is $\mathcal{R}_3$-rigid. While triangulations of the sphere have exactly the necessary number of edges needed for $\mathcal{R}_3$-rigidity, triangulations of surfaces of positive genus have more edges than are needed. 
Kastis and Power \cite{K-P} gave a precise characterisation for any graph embeddable on the projective plane and Cruickshank, Kitson and Power \cite{CKP} gave a characterisation for a specific class of toroidal graphs. Another related result was obtained by Nevo \cite{Nevo} who proved that all $K_5$-minor free graphs are $\mathcal{R}_3$-independent (and showed the analogous result holds in 4 and 5-dimensions).

We take a different direction and consider four generalisations of planar graphs and $\mathcal{R}_3$-rigidity in these contexts.

Let $G=(V,E)$ be a graph. We say that $G$ is:
\begin{itemize}
\item \emph{apex} if there exists $v\in V$ such that $G-v$ is planar.
\item \emph{critically apex} if, for all $v\in V$, $G-v$ is planar.
\item \emph{edge-apex} if there exists $e\in E$ such that $G-e$ is planar.
\item \emph{critically edge-apex} if, for all $e\in E$, $G-e$ is planar.
\end{itemize}

\Cref{fig:veapex} shows an apex graph (left) as well as an edge-apex graph (right) with the apex vertex/edge labeled. Further \Cref{tab:all-apex} illustrates the number of non-planar graphs in these families on small vertex sets.

\begin{figure}[ht]
    \centering
    \begin{tikzpicture}[scale=0.9]
        \foreach \w in {0,1,...,4}
        {
            \node[vertex] (a\w) at (72*\w+18:1.75) {};
            \node[vertex] (b\w) at (72*\w+18:1) {};
            \draw[edge] (a\w)--(b\w);
        }
        \foreach \w [remember=\w as \ow (initially 4)] in {0,1,...,4}
        {
            \draw[edge] (a\w)--(a\ow);
        }
        \foreach \w [remember=\w as \ow (initially 3)] in {0,2,4,1,3}
        {
            \draw[edge] (b\w)--(b\ow);
        }
        \node[vertex] (c1) at ($(a0)+(a4)-(a3)$) {};
        \node[vertex] (c2) at ($(c1)+(a3)-(a2)$) {};
        \node[vertex] (c3) at ($(c2)+(a2)-(a1)$) {};
        \node[vertex] (c4) at (a4) {};
        \node[hvertex] (c0) at (a0) {};
        \foreach \w [remember=\w as \ow (initially 4)] in {0,1,...,4}
        {
            \draw[edge] (c\w)--(c\ow);
        }
        \foreach \w [remember=\w as \ow (initially 3)] in {0,2,4,1,3}
        {
            \draw[edge] (c\w)--(c\ow);
        }
    \end{tikzpicture}
    \qquad
    \begin{tikzpicture}
        \clip[] (-2.2,-2)rectangle(2.2,2.2);
        \foreach \w in {0,1,...,11}
        {
            \node[vertex] (a\w) at (30*\w:1.5) {};
        }
        \foreach \w [remember=\w as \ow (initially 11)] in {0,1,...,11}
        {
            \draw[edge] (a\w)--(a\ow);
        }
        \foreach \w [remember=\w as \ow (initially 10)] in {0,2,4,...,10}
        {
            \draw[edge] (a\w)--(a\ow);
        }
        \node[vertex] (b1) at (-0.4,0.25) {};
        \node[vertex] (b2) at (0.4,0.25) {};
        \draw[edge] (b1)--(a2) (b1)--(a8) (a2)--(a8);
        \draw[edge] (b2)--(a4) (b2)--(a10) (a4)--(a10);
        \draw[hedge] (a0)to[bend right=140,looseness=3] node[labelsty,above] {$e$} (a6);
        \node[label={[labelsty]right:$v$}] at (a0) {}; 
        \node[label={[labelsty]left:$u$}] at (a6) {}; 
    \end{tikzpicture}
    \caption{Examples of apex and edge-apex graphs.}
    \label{fig:veapex}
\end{figure}
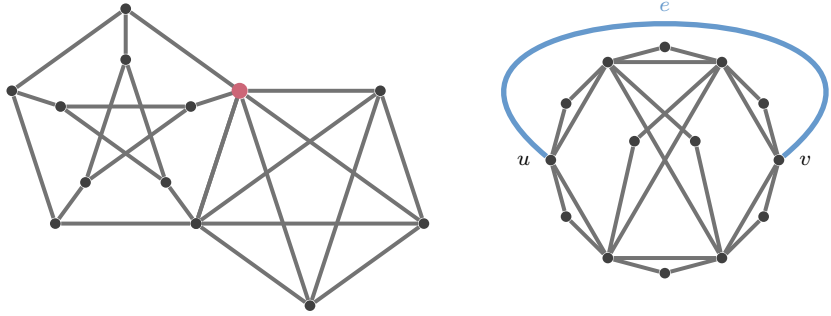

\begin{table}[ht]
    \centering
    \begin{tabular}{rrrrrr}
         \toprule
         $|V|$ & non-planar & apex & critically apex & edge-apex & critically edge-apex\\\midrule
         5 & 1 & 1 & 1 & 1 & 1\\
         6 & 13 & 12 & 8 & 11 & 2\\
         7 & 207 & 190 & 40 & 156 & 4 \\
         8 & 5143 & 4482 & 258 & 3398 & 10\\
         9 & 189195 & 142142 & 1310 & 89085 & 24\\
         10 & 10663766 & 5517578 & 6084 & 2559911 & 51\\
         \bottomrule
    \end{tabular}
    \caption{Number of non-planar connected graphs with different apex properties. The data behind the statistics in this table is available \cite{DataApex}.}
    \label{tab:all-apex}
\end{table}

These terms generalise in the obvious way. That is, we can consider \emph{$k$-apex, critically $k$-apex, $k$-edge-apex and critically $k$-edge-apex} graphs, for $k\in \mathbb{N}$, by taking the case $k=1$ to be as above and the generalisation to delete sets of vertices/edges of size $k$. 
The complete graph $K_n$, $n \geq 5$, is obviously $(n-4)$-apex and an easy calculation shows that the edge apicity of $K_n$ is $\binom{n-3}{2}$.

These concepts, and variants thereof, have been studied in the graph theory literature \cite{Ding,gubser_1996,variations}, in particular apex graphs seem to be fundamental \cite{Jorg,RST} in areas of graph theory related to Hadwiger's conjecture. For a further example, non-planar critically apex graphs have been called critical non-planar graphs and the number of these on small vertex sets can be found at OEIS \cite[A158922]{OEIS}. 

Note that if a graph $G$ is edge-apex, then it is trivially also vertex-apex.
Also, planar graphs trivially satisfy these definitions but each of the four graph classes is strictly stronger than planar graphs and none of the classes coincide.

As well as $\mathcal{R}_3$-independence we also provide new results for other rigidity concepts. A framework $(G,p)$ is \emph{globally $\mathcal{R}_d$-rigid} if every other framework $(G,q)$ in $\mathbb{R}^d$ with the same edge lengths can be obtained from $(G,p)$ by an isometry of $\mathbb{R}^d$. For details on the theory of global rigidity the reader is directed to \cite{gortler2010characterizing,JacksonJordan,Tsuff}.
In the context of global rigidity, triangulations of surfaces have received attention recently \cite{CJTglob,JT19}. We also provide an application of our results to Gaussian graphical models using a recent link between rigidity theory and maximum likelihood estimation \cite{Betal,gross2018maximum}.

We conclude the introduction by outlining the structure of the paper, highlighting our main contributions. In \Cref{sec:background} we provide the background results from topological graph theory and rigidity theory that we need. 
\Cref{sec:edge-apex,sec:apex} contain key results characterising $\mathcal{R}_3$-independence for edge-apex graphs (\Cref{t:ear}) and providing structural results on the apex case (\Cref{prop:apex} and \Cref{prop:apex2}).
In \Cref{sec:critical} we characterise $\mathcal{R}_3$-independence for critically apex graphs (\Cref{thm:critapex}) and critically 2-apex graphs (\Cref{thm:crit2apex}). For critically edge-apex graphs we characterise $\mathcal{R}_3$-independence for $k$-edge-apex graphs for all $k\leq 7$ (\Cref{thm:critkedgeapex}).
Then in \Cref{sec:global} we analyse extensions to global rigidity including a characterisation for edge-apex graphs (\Cref{t:eagr}). We also consider extensions to maximum likelihood thresholds for Gaussian graphical models in \Cref{sec:mlt}. Here our key results include \Cref{thm:mltea} which shows two natural graph parameters coincide for edge-apex graphs and \Cref{apexmlt} which tightly bounds the same parameters for $k$-apex graphs. 

We have just touched the surface of an investigation of rigidity properties for nearly planar graphs and we hope to have interested the reader in continuing this line of investigation.

\section{Background}
\label{sec:background}
In this section we provide the background results from topological
graph theory and rigidity theory that we need later.

\subsection{Planar graphs}

A graph $G$ is planar if it can be drawn in the plane without edge crossings. We will consider simple connected planar graphs. From F\'{a}ry's theorem~\cite{Fary} we know that every such graph has an embedding such that the edges are non-crossing straight line segments and all faces are (topological) disks. An edge-maximal planar graph is a \emph{triangulation} and all faces of such a graph are 3-cycles.
By Euler's formula, every triangulation on $n$ vertices must have exactly $3n-6$ edges.
A triangulation is 3-connected unless it is $K_3$. 
Given two connected disjoint planar graphs $G_1$, $G_2$, we can get a connected planar graph by identifying  vertex $a$  of $G_1$  with  vertex $b$ of $G_2$. We call this operation vertex join and denote the resulting graph a $G_1 \dot{v} G_2$, where $v$ is understood to be the vertex $a=b$.  Also identifying a pair of edges yields a planar graph (edge join $G_1 \bar{e} G_2$). The resulting planar graphs are never edge maximal.
However, we can take two triangulations $T_1$ and $T_2$ and identify a facial triangle of $T_1$ with a facial triangle of $T_2$ to obtain a triangulation $T_1 \binom{\Delta}{f} T_2$ inheriting all faces from $T_1$ and $T_2$, only the identified face is not  a face any longer, it is a non-facial 3-cycle. We call this operation a $\Delta$ join. 

\begin{lemma}\label{structure}
  Every simple planar triangulation $T$ can be uniquely written as $ T= T_1\Delta T_2 \Delta \ldots \Delta T_k$, where each $T_i$ is $K_4$ or a 4-connected triangulation. The graph having $T_i$ as its vertices  and as its edges the faces to be joined is a tree, denoted $S(T)$.   
\end{lemma}

\begin{proof}
A classical result of Whitney shows that every 3-connected planar graph has an essentially unique embedding in the plane. So given a simple planar triangulation $T$, which 3-cycles are faces and which are not facial is determined. If $T$ contains no non-facial 3-cycle, it is 4-connected. If there is a non-facial 3-cycle $f$, we can write $T=T_1 \binom{\Delta}{f} T_2$
where $T_i$ is a triangulation with a face $f$. Note that two non-facial triangles can have at most one edge in common, so every other non-facial triangle of $T$ is contained in either $T_1$ or $T_2$, but not both.
Let $k$ be the number of non-facial 3-cycles of $T$, so we obtain $k+1$ 4-connected triangulations (resp.\ tetrahedra) $T_0 \ldots T_k$. Now consider the graph $S(T)$ whose vertices are the $T_i$ and $T_i$ is adjacent to $T_j$ if both contain a face $f$ which is a non-facial 3-cycle of $T$. Then $S(T)$ is connected and hence a tree. 
\end{proof}

We call $S(T)$ the \emph{structure tree} of the triangulation $T$ and its vertices the 4-blocks of $T$.

\subsection{Graphs on surfaces of higher genus}

A map is a graph embedded into a compact connected two-dimensional manifold such that the complement of the graph in the manifold is a disjoint union of open topological disks called faces. 
A 1-vertex 1-face map on a surface of genus $g$ must have, by Euler's formula, $2-2g$ edges if the surface is orientable, or $2-g$ edges in the unorientable case, so for a graph to be embeddable on a surface such that every face is a (topological) disk, a graph must have a certain minimum number of edges. On the other hand, if we want every face to be a 3-cycle, then the number of edges is exactly $3n-6+2g$ if the surface is orientable and $3n-6 +g$ in the non-orientable case.

The term Euler-genus was coined in~\cite{EG}, it is called generalized genus in~\cite{Add}.
If a surface $\Sigma$ is obtained from the sphere by the addition of $h$ handles
and $k$ crosscaps, then the \emph{Euler-genus} of $\Sigma$, $eg (\Sigma)$, is defined to be $k + 2h$. For a connected graph $G$, the Euler genus of G is 
the minimum of $eg (\Sigma)$ over all $ \Sigma$ in which $G$ can be cellularly embedded. 

$K_5$ is not planar, but it can be cellularly embedded on the projective plane as well as the torus. The projective plane has Euler-genus~1, while the torus has Euler-genus~2, so the Euler-genus of $K_5$ is~1. A projective plane embedding of $K_5$ is shown in \Cref{oriented09}.

\begin{figure}[htb]
\centering
\begin{tikzpicture}
    \begin{scope}
        \foreach \l [count=\i from 0] in {c,d,e,a,b,c,d,e,a,b}
        {
            \node[vertex,label={[labelsty,label distance=-2pt]90+\i*36:$\l$}] (v\i) at (90+\i*36:1.2) {};
        }
        \foreach \i [remember=\i as \o (initially 9)] in {0,1,...,9}
        {
            \draw[edge] (v\i)--(v\o);
        }
        \foreach \i [remember=\i as \o (initially 9)] in {1,3,...,9}
        {
            \draw[edge] (v\i)--(v\o);
        }
    \end{scope}
    \node[font=\Large] at (2,0) {$+$};
    \begin{scope}[xshift=4cm]
        \foreach \l [count=\i from 0] in {c,b,a,e,d,c,b,a,e,d}
        {
            \node[vertex,label={[labelsty,label distance=-2pt]90+\i*36:$\l'$}] (w\i) at (90+\i*36:1.2) {};
        }
        \foreach \i [remember=\i as \o (initially 9)] in {0,1,...,9}
        {
            \draw[edge] (w\i)--(w\o);
        }
         \foreach \i [remember=\i as \o (initially 8)] in {0,2,...,8}
        {
            \draw[edge] (w\i)--(w\o);
        }
    \end{scope}
    \node[font=\Large] at (6,0) {$=$};
    \begin{scope}[xshift=8cm]
        \begin{scope}
            \foreach \l [count=\i from 0] in {c,d,e,a,b,c,d,,,b}
            {
                \node[vertex,label={[labelsty,label distance=-2pt]90+\i*36:$\l$}] (v\i) at (90+\i*36:1.2) {};
            }
            \foreach \i [remember=\i as \o (initially 8)] in {9,0,1,2,3,4,5,6,7}
            {
                \draw[edge] (v\i)--(v\o);
            }
            \foreach \i [remember=\i as \o (initially 9)] in {1,3,...,9}
            {
                \draw[edge] (v\i)--(v\o);
            }
        \end{scope}
        \begin{scope}[shift={($(v7)+(v8)$)}]
            \foreach \l [count=\i from 0] in {c',b',,,d',c',b',a',e',d'}
            {
                \node[vertex,label={[labelsty,label distance=-2pt]90+\i*36:$\l$}] (w\i) at (90+\i*36:1.2) {};
            }
            \foreach \i [remember=\i as \o (initially 3)] in {4,5,6,7,8,9,0,1,2}
            {
                \draw[edge] (w\i)--(w\o);
            }
             \foreach \i [remember=\i as \o (initially 8)] in {0,2,...,8}
            {
                \draw[edge] (w\i)--(w\o);
            }
        \end{scope}
        \node[vertex,label={[labelsty,label distance=-1pt]180:$a$}] at (v7) {};
        \node[vertex,label={[labelsty,label distance=-1pt]0:$e$}] at (v8) {};
    \end{scope}
\end{tikzpicture}
\caption{$K_5$ on the projective plane and $K_5 \overline{(a,e)} K_5$ on the Klein Bottle. \label{oriented09}}
\end{figure}
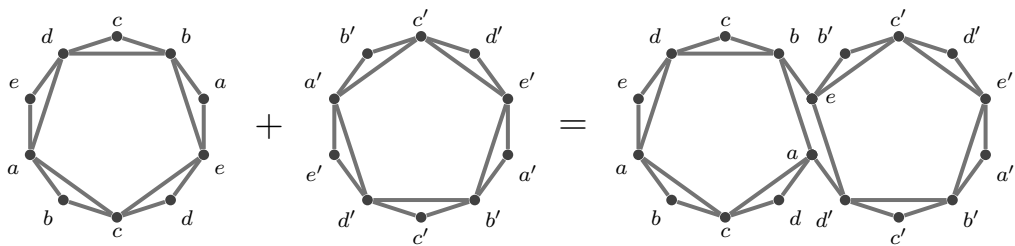

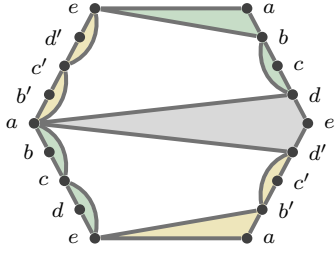
\begin{figure}[htb]
\centering
\begin{tikzpicture}[yscale=1.1,baseline={(0,0)}]
    \node[vertex,label={[labelsty]180:$a$}] (a) at (0,0) {};
    \node[vertex,label={[labelsty]180:$b'$}] (b) at (60:0.4) {};
    \node[vertex,label={[labelsty]180:$c'$}] (c) at (60:0.8) {};
    \node[vertex,label={[labelsty]180:$d'$}] (d) at (60:1.2) {};
    \node[vertex,label={[labelsty]180:$e$}] (e) at (60:1.6) {};
    \node[vertex,label={[labelsty]0:$a$}] (a2) at ($(e)+(2,0)$) {};
    \node[vertex,label={[labelsty]0:$b$}] (b2) at ($(a2)+(-60:0.4)$) {};
    \node[vertex,label={[labelsty]0:$c$}] (c2) at ($(a2)+(-60:0.8)$) {};
    \node[vertex,label={[labelsty]0:$d$}] (d2) at ($(a2)+(-60:1.2)$) {};
    \node[vertex,label={[labelsty]0:$e$}] (e2) at ($(a2)+(-60:1.6)$) {};
    \node[vertex,label={[labelsty]180:$b$}] (br) at (-60:0.4) {};
    \node[vertex,label={[labelsty]180:$c$}] (cr) at (-60:0.8) {};
    \node[vertex,label={[labelsty]180:$d$}] (dr) at (-60:1.2) {};
    \node[vertex,label={[labelsty]180:$e$}] (er) at (-60:1.6) {};
    \node[vertex,label={[labelsty]0:$a$}] (ar2) at ($(er)+(2,0)$) {};
    \node[vertex,label={[labelsty]0:$b'$}] (br2) at ($(ar2)+(60:0.4)$) {};
    \node[vertex,label={[labelsty]0:$c'$}] (cr2) at ($(ar2)+(60:0.8)$) {};
    \node[vertex,label={[labelsty]0:$d'$}] (dr2) at ($(ar2)+(60:1.2)$) {};
    
    \draw[edge] (a)--(b) (b)--(c) (c)--(d) (d)--(e);
    \draw[edge] (a2)--(b2) (b2)--(c2) (c2)--(d2) (d2)--(e2);
    \draw[edge] (a)--(d2) (e)--(a2) (e)--(b2);
    \draw[edge] (a)to[bend right=40] (c);
    \draw[edge] (c)to[bend right=40] (e);
    \draw[edge] (b2)to[bend right=40] (d2);
    \draw[edge] (a)--(br) (br)--(cr) (cr)--(dr) (dr)--(er);
    \draw[edge] (ar2)--(br2) (br2)--(cr2) (cr2)--(dr2) (dr2)--(e2);
    \draw[edge] (a)--(dr2) (er)--(ar2) (er)--(br2);
    \draw[edge] (a)to[bend left=40] (cr);
    \draw[edge] (cr)to[bend left=40] (er);
    \draw[edge] (br2)to[bend left=40] (dr2);

    \begin{scope}[on background layer]
        \fill[colGl] (e.center)--(a2.center)--(b2.center)--(e.center)--cycle;
        \fill[colGl] (b2.center)--(c2.center)--(d2.center)to[bend left=40](b2.center)--cycle;
        \fill[colOl] (a.center)--(b.center)--(c.center)to[bend left=40](a.center)--cycle;
        \fill[colOl] (c.center)--(d.center)--(e.center)to[bend left=40](c.center)--cycle;
        \fill[colGrayl] (a.center)--(d2.center)--(e2.center)--(dr2.center)--(a.center)--cycle;
        \fill[colOl] (er.center)--(ar2.center)--(br2.center)--(er.center)--cycle;
        \fill[colOl] (br2.center)--(cr2.center)--(dr2.center)to[bend right=40](br2.center)--cycle;
        \fill[colGl] (a.center)--(br.center)--(cr.center)to[bend right=40](a.center)--cycle;
        \fill[colGl] (cr.center)--(dr.center)--(er.center)to[bend right=40](cr.center)--cycle;
    \end{scope}
\end{tikzpicture}
\caption{Torus map of $K_5 \overline{(a,e)} K_5$.}
\end{figure}

While the operations $\dot{v}$ and $\bar{e}$ performed on embedded graphs of two surfaces of genus~$\geq 1$ do not directly yield embeddings on some surface, we can perform the $\Delta$ join. Consider two surfaces $\Sigma_1$ and $\Sigma_2$ of Euler genus $k_1$ and $k_2$ respectively and a simple triangulation $T_1$ on $\Sigma_1$ and a simple triangulation  $T_2$ of $\Sigma_2$. If $T_i$ has $n_i$ vertices and $3n_i-6+k_i$  then $T_1 \Delta T_2 $ has $n_1 +n_2 -3$ vertices and 
$3(n_1 + n_2 -3) -6 + k_1+k_2$ edges, which means it is a triangulation of a surface of Euler genus $k_1+k_2$.

We will use the following theorem from topological graph theory which is a special case of \cite[Theorem 1]{Add}.

\begin{theorem}\label{addgenus}
Suppose $G=(V,E)$.
Let $G_1=(V_1,E_1)$ and $G_2=(V_2,E_2)$ be graphs of Euler genus $eg_1$ and $eg_2$ such that $V=V_1\cup V_2$, $E=E_1\cup E_2$, $|V_1\cap V_2|=2$ and $|E_1\cap E_2|=1$.
Then the Euler genus of $G$ is $eg_1+eg_2$.  
\end{theorem}

\subsection{Rigidity of graphs}

Given a simple graph $G$, we consider its vertex set embedded in $\mathbb{R}^d$ and interpret its edges as length constraints. If edge lengths are fixed under any motion of the vertices, we get a system of quadratic equations. Regarding the coordinates of the vertices as differentiable function of time we can take the derivative of the system of constraint equations. The Jacobian matrix (up to scaling) is called the  \emph{rigidity matrix} $R(G,p)$. $R(G,p)$ is an $|E|\times d|V|$ matrix  in which, for $e=v_iv_j\in E$, the submatrices in row $e$ and columns $v_i$ and $v_j$ are $p(v_i)-p(v_j)$ and $p(v_j)-p(v_i)$, respectively, and all other entries are zero. 
We say that $(G,p)$ is \emph{infinitesimally rigid} if $|V|\leq d+1$ and $\rank\, R(G,p)=\binom{|V|}{2}$ or $|V|\geq d+2$ and $\rank\, R(G,p)=d|V|-\binom{d+1}{2}$.

The \emph{$d$-dimensional rigidity matroid} of a graph $G=(V,E)$ is the matroid $\mathcal{R}_d(G)$ on $E$ in which a set of edges $F\subseteq E$ is independent whenever the corresponding rows of $R(G,p)$ are linearly independent, for some (or equivalently every) generic $p$. 
We use the terms $\mathcal{R}_d$-\emph{independent} and $\mathcal{R}_d$-\emph{circuit} to refer to the graphs induced by independent sets and circuits of $\mathcal{R}_d$, and $r_d$ for its rank function. We also say that $G$ is $\mathcal{R}_d$-\emph{rigid} if $r_d(G)=d|V|-\binom{d+1}{2}$.

In what follows we need some well known results about $\mathcal{R}_d$-independence. The first is an easy necessary condition. A graph $G=(V,E)$ is $(k,l)$-\emph{sparse} if for every subset $X\subset V$, with at least $k$ elements, the number of edges in the subgraph of $G$ induced by $X$ is at most $k|X|-l$. If $G$ is $(k,l)$-sparse and $|E|=k|V|-l$ then $G$ is said to be $(k,l)$-\emph{tight}.

\begin{lemma}[{\cite{Maxwell}}]\label{lem:max}
Let $G$ be $\mathcal{R}_d$-independent. Then $G$ is $(d,\binom{d+1}{2})$-sparse.  
\end{lemma}

This lemma gives a simple necessary condition that narrows down the number of non-planar graphs in each of our apex families. Comparing \Cref{tab:all-apex} with \Cref{tab:sparse-apex}, which gives statistics for non-planar $(3,6)$-sparse graphs, illustrates. For example, $82.55\%$ of the 10 vertex apex graphs are $(3,6)$-sparse.

\begin{table}[ht]
    \centering
    \begin{tabular}{rrrrrr}
         \toprule
         $|V|$ & non-planar & apex & critically apex & edge-apex & critically edge-apex\\\midrule
         6 & 7 & 7 & 7 & 7 & 2 \\
         7 & 135 & 133 & 39 & 121 & 4 \\
         8 & 3637 & 3512 & 257 & 3000 & 10 \\
         9 & 128411 & 115999 & 1309 & 83349 & 24 \\
         10 & 6003893 & 4555219 & 6083 & 2463215 & 51 \\
         \bottomrule
    \end{tabular}
    \caption{Number of non-planar connected $(3,6)$-sparse graphs with different apex properties \cite{DataApexSparse}.}
    \label{tab:sparse-apex}
\end{table}

A graph $G'$ is said to be obtained from another graph $G$ by a \emph{($d$-dimensional) 0-extension}
if $G=G'-v$ for a vertex  $v\in V(G')$ with $d_{G'}(v)=d$; or a \emph{($d$-dimensional) 1-extension} if $G=G'-v+xy$ for a vertex $v\in V(G')$ with $d_{G'}(v)=d+1$ and $x,y\in N_{G'}(v)$.
The inverse operations of 0-extension and 1-extension are called \emph{0-reduction} and \emph{1-reduction}, respectively.

\begin{lemma}[{\cite[Lemma 11.1.1, Theorem 11.1.7]{Wlong}}]\label{lem:01ext}
Let $G$ be $\mathcal{R}_d$-independent and let $G'$ be obtained from $G$ by a 0-extension or a 1-extension. Then $G'$ is $\mathcal{R}_d$-independent.
\end{lemma}

More lemmas we will need are as follows.

\begin{lemma}[{\cite[Lemma 11.1.9]{Wlong}}]\label{lem:intbridge}
Let $G_1$, $G_2$ be subgraphs of a graph $G$ and suppose that $G=G_1\cup G_2$.
\begin{enumerate}
\item\label{it:intbridge:rig} 
	If $|V(G_1)\cap V(G_2)|\geq d$ and $G_1,G_2$ are $\mathcal{R}_d$-rigid then $G$ is $\mathcal{R}_d$-rigid.
\item\label{it:intbridge:indep} 
	If $G_1\cap  G_2$ is $\mathcal{R}_d$-rigid and $G_1,G_2$ are $\mathcal{R}_d$-independent then $G$ is $\mathcal{R}_d$-independent.
\item\label{it:intbridge:rank} 
	If $|V(G_1)\cap V(G_2)| \leq d-1$, $u\in V(G_1)-V(G_2)$ and $v\in V(G_2)-V(G_1)$ then
	$r_d(G+uv)=r_d(G)+1$.
\end{enumerate}
\end{lemma}

\begin{lemma}[{\cite{Coning,GGJ}}]\label{lem:cone}
Let $G'$ be the cone of a graph $G=(V,E)$. That is $G'$ is obtained from $G$ by adding one new vertex and joining it to every vertex of $G$. Then:
\begin{enumerate}
    \item $G$ is $\mathcal{R}_d$-independent if and only if $G'$ is $\mathcal{R}_{d+1}$-independent. 
    \item $G$ is a $\mathcal{R}_d$-circuit if and only if $G'$ is a $\mathcal{R}_{d+1}$-circuit. 
\end{enumerate}   
\end{lemma}
 
An \emph{equilibrium stress} $\omega$ of a framework $(G,p)$ is a vector in the cokernel of $R(G,p)$. In other words it is an assignment of weights to the edges of $G$ that are in equilibrium in the framework $(G,p)$. Using stresses we can say that $G$ is $\mathcal{R}_d$-independent if it has no equilibrium stress and an $\mathcal{R}_d$-circuit if it has a unique equilibrium stress (up to scale) which is non-zero on every edge of $G$.

An equilibrium stress gives rise to the \emph{stress matrix} $\Omega$, a weighted laplacian matrix where the off-diagonal $(i,j)$-entry is the negative of the element of $\omega$ corresponding to the edge $v_iv_j$ (or 0 if that edge is not present in $G$) and the diagonal is the sum of the weights for edges incident to the vertex. We say that $G$ admits: a \emph{full rank stress} if there exists a generic framework $(G,p)$ in which some non-zero equilibriumn stress gives rise to a stress matrix of rank $|V|-d-1$; and a
\emph{PSD stress} if there exists a generic framework $(G,p)$ in which some non-zero equilibriumn stress gives rise to a stress matrix which is positive semi-definite. 

Note that the first of these two properties is generic, that is if a generic framework $(G,p)$ has a full rank stress then so does any other generic framework of $G$, but the second is not a generic property (one generic framework of $G$ can be PSD and another generic framework of $G$ can be indefinite). We conclude this section by noting the fundamental result that global $\mathcal{R}_d$-rigidity is a generic property \cite{gortler2010characterizing}.

\section{Edge-apex graphs}
\label{sec:edge-apex}

We first prove a characterisation of $\mathcal{R}_3$-rigidity for edge-apex graphs.
We will make use of \Cref{structure} and the following theorem of Whiteley.

\begin{theorem}[{\cite[Theorem 5.3]{MR927685}}]\label{thm:whiteleytriangplusedge}
Let $G$ be a 4-connected graph obtained from a plane triangulation by adding one edge. Then $G$ is a rigid $\mathcal{R}_3$-circuit.   
\end{theorem}

\begin{theorem}\label{t:ear}
    Let $G$ be an edge-apex graph. Then $G$ is $\mathcal{R}_3$-independent if and only if it is $(3,6)$-sparse.
\end{theorem}

\begin{proof}
By \Cref{lem:max} it suffices to prove the sufficiency. 
Since $G$ is edge-apex, it either is planar, in which case it is $\mathcal{R}_3$-independent, or there is an edge $e$ such that $G-e$ is planar. Extend $G-e$ to an edge maximal planar graph by adding edges $e_1, e_2, \ldots e_k$. We denote this triangulation by $G^{\Delta} = G - e +e_1, \ldots, e_k$. 
Note that from $(3,6)$-sparsity we have that $k\geq 1$. We also know, by \Cref{thm:gluck}, that $G^{\Delta}$ is $\mathcal{R}_3$-independent and $\mathcal{R}_3$-rigid, so $G^{\Delta}+e$ contains a unique $\mathcal{R}_3$-circuit. Let $e=uv$. Decompose $G^{\Delta}$ into 4-blocks and identify a block $C_u$ that contains $u$ and a component $C_v$ that contains $v$. In the structure tree  $S(G^{\Delta })$ (from \Cref{structure}) there is a unique path from $C_u$ to $C_v$. 
If $u$ (resp. $v$) is contained in more than one component then we shorten this unique path if possible. Let $C'_u \ldots C'_v$ be the shortest such path in $S(G^{\Delta})$ (note that it could be of length~0), then $H=C'_u \Delta \ldots \Delta C'_v +e$ is a 4-connected subgraph of  $G^{\Delta} +e$ and $H-e$ is a triangulation. 
Hence, $H$ is an $\mathcal{R}_3$-circuit by \Cref{thm:whiteleytriangplusedge}. By $(3,6)$-sparsity, $H$ must contain at least one of the $e_i's$, which means that $G^{\Delta} +e-e_i$ is $\mathcal{R}_3$-independent and hence $G$ is $\mathcal{R}_3$-independent.
\end{proof}

\subsection[k-edge-apex graphs]{$k$-edge-apex graphs}

\begin{corollary}
Suppose that $G=(V,E)$ is $(3,6)$-tight and $k$-edge-apex for some $k\geq 1$. Then rank $R_3(G,p) \geq 3|V|-5-k$.
\end{corollary}

\begin{proof}
When $k=1$ the statement is a trivial reformulation of one direction of  \Cref{t:ear}.
Suppose the lemma holds for all $j<k$ and consider a $(3,6)$-tight $k$-edge-apex graph $G$ for some $k>1$. By definition $G$ contains some edge $e$ such that $G-e$ is $(k-1)$-edge-apex. Since there exists an edge $f$ such that $G-e+f$ is $(k-1)$-edge-apex and $(3,6)$-tight, the inductive hypothesis implies that 
\begin{equation*}
\rank R_3(G-e,p)\geq 3|V|-5-(k-1)-1=3|V|-5-k.\qedhere    
\end{equation*}
\end{proof}

Similarly one can establish the slightly more general statement that if $G$ is $(3,6)$-sparse and $k$-edge-apex, then dim coker $R_3(G,p)\leq k-1$.
The corollary is, in a sense, best possible since it is not true that 2-edge-apex graphs are $\mathcal{R}_3$-independent if and only if they are $(3,6)$-sparse. We propose the following conjecture that would, in a reasonably precise sense, explain which 2-edge-apex $(3,6)$-sparse graphs are $\mathcal{R}_3$-independent.

\begin{conjecture}
The following hold.
\begin{enumerate}
    \item If $G$ is planar and $G+\{e,f\}$ is a $(3,6)$-tight flexible $\mathcal{R}_3$-circuit then $G$ is the 2-sum of two rigid $\mathcal{R}_3$-circuits.
    \item If $G$ is planar and $G+\{e,f\}$ is a flexible $\mathcal{R}_3$-circuit then $G+\{e,f\}$ is $(3,6)$-tight.
\end{enumerate}
\end{conjecture}

Suppose $G$ is planar and $G+\{e,f\}$ is $(3,6)$-tight. Then $G+e$ is $\mathcal{R}_3$-independent by \Cref{t:ear}. Hence, either $G+\{e,f\}$ is $\mathcal{R}_3$-rigid or it has exactly one non-trivial infinitesimal motion. Hence, $G+\{e,f\}$ contains a unique $\mathcal{R}_3$-circuit. 

    \begin{figure}[ht]
        \centering
        \begin{tikzpicture}
          \coordinate (rr1) at (0,1.5);
          \coordinate (rr2) at (0,-1.5);          
          \node[vertex] (r1) at (rr1) {};
          \node[vertex] (r2) at (rr2) {};
          \node[vertex] (a1) at (-1.1,-0.06) {};
          \node[vertex] (a2) at (1.1,-0.06) {};
          \node[vertex] (b1) at (-0.4,-0.12) {};
          \node[vertex] (b2) at (0.4,-0.12) {};
          \node[vertex] (c1) at (-0.7,-0.3) {};
          \node[vertex] (c2) at (0.7,-0.3) {};
          \draw[edge] (r1)edge(a1) (r1)edge(b1) (r1)edge(c1) (r1)edge(a2) (r1)edge(b2) (r1)edge(c2);
          \draw[edge] (r2)edge(a1) (r2)edge(b1) (r2)edge(c1) (r2)edge(a2) (r2)edge(b2) (r2)edge(c2);
          \draw[edge] (a1)edge(b1) (a1)edge(c1) (b1)edge(c1);
          \draw[edge] (a2)edge(b2) (a2)edge(c2) (b2)edge(c2);
        \end{tikzpicture}   
    \caption{The double banana \dbgraph.}
    \label{fig:doublebanana}
\end{figure}
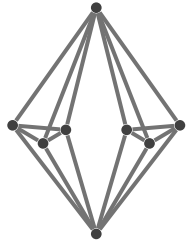

To explain the strategy used to prove  \Cref{t:ear}
and possibly attack the conjecture, we provide the following example. We start with the well known double banana graph \dbgraph, depicted in \Cref{fig:doublebanana}, which is the smallest $(3,6)$-sparse graph that is $
\mathcal{R}_3$-dependent.
\Cref{structtree01fig} shows \dbgraph\ with edges $ac$ and $ac'$ removed and embedded in the plane. Since the graph is not 3-connected, the embedding is not unique. Fix the embedding shown and extend the graph to a triangulation by adding the red edges. We then decompose the triangulation into 4-blocks and see that the structure tree is a path on 5 vertices. Adding the edge $ac$ forms a $\mathcal{R}_3$-circuit with using just two of the blocks, it is a $K_5$. Similarly the edge $ac'$ creates another $K_5$ using two different blocks. Now we have two $\mathcal{R}_3$-circuits intersecting in exactly one red edge and we found a copy of \dbgraph, a $\mathcal{R}_3$-circuit.

\begin{figure}[htb]
\centering
\begin{tikzpicture}
    \begin{scope}
        \node[vertex,label={[labelsty]90:$a$}] (a) at (0,0) {};
        \node[vertex,label={[labelsty]-90:$b$}] (b) at (0,-2) {};
        \node[vertex,label={[labelsty,label distance=-2pt]90:$c$}] (c) at (210:0.8) {};
        \node[vertex,label={[labelsty,label distance=-2pt]90:$c'$}] (cs) at (-30:0.8) {};
        \node[vertex,label={[labelsty,label distance=-2pt]150:$d$}] (d) at (150:2) {};
        \node[vertex,label={[labelsty,label distance=-2pt]30:$d'$}] (ds) at (30:2) {};
        \node[vertex,label={[labelsty]90:$e$}] (e) at (210:0.5) {};
        \node[vertex,label={[labelsty]90:$e'$}] (es) at (-30:0.5) {};
        \draw[edge,colR] (a)--(b) (d)--(ds);
        \draw[edge] (a)--(d) (a)--(ds) (a)--(e) (a)--(es) (b)--(c) (b)--(cs) (b)--(d) (b)--(ds) (b)--(e) (b)--(es) (c)--(d) (cs)--(ds) (c)--(e) (cs)--(es) (e)--(d) (es)--(ds);
    \end{scope}
    \begin{scope}[yshift=-3.5cm]
        \begin{scope}[xshift=-5cm]
            \node[vertex,label={[labelsty,label distance=-2pt]90:$c$}] (c) at (0,0) {};
            \node[vertex,label={[labelsty,label distance=-2pt]-90:$b$}] (b) at (0,-1) {};
            \node[vertex,label={[labelsty,label distance=-4pt]150:$d$}] (tl) at (150:1) {};
            \node[vertex,label={[labelsty,label distance=-4pt]30:$e$}] (tr) at (30:1) {};
            \draw[edge] (c)--(b) (c)--(tl) (c)--(tr) (tr)--(tl) (tl)--(b) (b)--(tr);
        \end{scope}
        \begin{scope}[xshift=-2.5cm]
            \node[vertex,label={[labelsty,label distance=-2pt]90:$e$}] (c) at (0,0) {};
            \node[vertex,label={[labelsty,label distance=-2pt]-90:$b$}] (b) at (0,-1) {};
            \node[vertex,label={[labelsty,label distance=-4pt]150:$d$}] (tl) at (150:1) {};
            \node[vertex,label={[labelsty,label distance=-4pt]30:$a$}] (tr) at (30:1) {};
            \draw[edge] (c)--(b) (c)--(tl) (c)--(tr) (tr)--(tl) (tl)--(b);
            \draw[edge,colR] (b)--(tr);
        \end{scope}
        \begin{scope}[xshift=-0cm]
            \node[vertex,label={[labelsty,label distance=-2pt]90:$a$}] (c) at (0,0) {};
            \node[vertex,label={[labelsty,label distance=-2pt]-90:$b$}] (b) at (0,-1) {};
            \node[vertex,label={[labelsty,label distance=-4pt]150:$d$}] (tl) at (150:1) {};
            \node[vertex,label={[labelsty,label distance=-4pt]30:$d'$}] (tr) at (30:1) {};
            \draw[edge] (c)--(tl) (c)--(tr) (tl)--(b) (b)--(tr);
            \draw[edge,colR] (c)--(b) (tr)--(tl);
        \end{scope}
        \begin{scope}[xshift=2.5cm]
            \node[vertex,label={[labelsty,label distance=-2pt]90:$e'$}] (c) at (0,0) {};
            \node[vertex,label={[labelsty,label distance=-2pt]-90:$b$}] (b) at (0,-1) {};
            \node[vertex,label={[labelsty,label distance=-4pt]150:$a$}] (tl) at (150:1) {};
            \node[vertex,label={[labelsty,label distance=-4pt]30:$d'$}] (tr) at (30:1) {};
            \draw[edge] (c)--(b) (c)--(tl) (c)--(tr) (tr)--(tl) (b)--(tr);
            \draw[edge,colR] (tl)--(b);
        \end{scope}
        \begin{scope}[xshift=5cm]
            \node[vertex,label={[labelsty,label distance=-2pt]90:$c'$}] (c) at (0,0) {};
            \node[vertex,label={[labelsty,label distance=-2pt]-90:$b$}] (b) at (0,-1) {};
            \node[vertex,label={[labelsty,label distance=-4pt]150:$e'$}] (tl) at (150:1) {};
            \node[vertex,label={[labelsty,label distance=-4pt]30:$d'$}] (tr) at (30:1) {};
            \draw[edge] (c)--(b) (c)--(tl) (c)--(tr) (tr)--(tl) (tl)--(b) (b)--(tr);
        \end{scope}
        \foreach \x/\l in {-3.75/bed,-1.25/bad,1.25/bad',3.75/be'd'}
        {
            \draw[{Latex[round]}-{Latex[round]}] ($(\x-0.5,-1)$)to node[labelsty,above] {$\l$} ($(\x+0.5,-1)$);
        }
    \end{scope}
\end{tikzpicture}
\caption{A triangulation and its decomposition into 4-blocks.}\label{structtree01fig}
\end{figure}
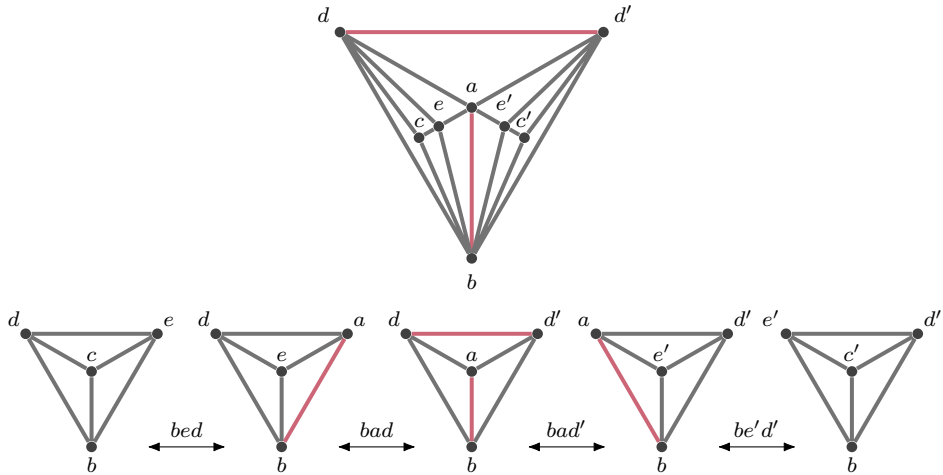

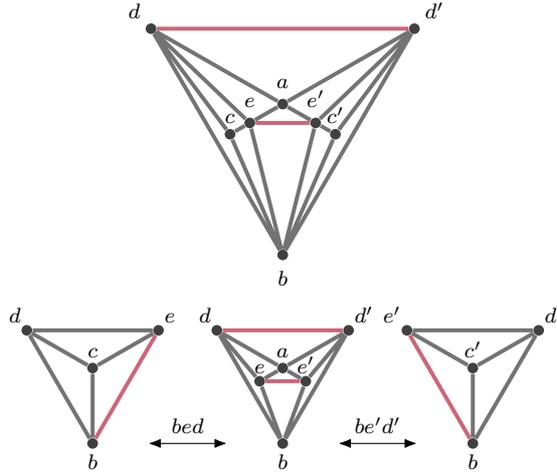
\begin{figure}[htb]
\centering
\begin{tikzpicture}
    \begin{scope}
        \node[vertex,label={[labelsty]90:$a$}] (a) at (0,0) {};
        \node[vertex,label={[labelsty]-90:$b$}] (b) at (0,-2) {};
        \node[vertex,label={[labelsty,label distance=-2pt]90:$c$}] (c) at (210:0.8) {};
        \node[vertex,label={[labelsty,label distance=-2pt]90:$c'$}] (cs) at (-30:0.8) {};
        \node[vertex,label={[labelsty,label distance=-2pt]150:$d$}] (d) at (150:2) {};
        \node[vertex,label={[labelsty,label distance=-2pt]30:$d'$}] (ds) at (30:2) {};
        \node[vertex,label={[labelsty]90:$e$}] (e) at (210:0.5) {};
        \node[vertex,label={[labelsty]90:$e'$}] (es) at (-30:0.5) {};
        \draw[edge,colR] (e)--(es) (d)--(ds);
        \draw[edge] (a)--(d) (a)--(ds) (a)--(e) (a)--(es) (b)--(c) (b)--(cs) (b)--(d) (b)--(ds) (b)--(e) (b)--(es) (c)--(d) (cs)--(ds) (c)--(e) (cs)--(es) (e)--(d) (es)--(ds);
    \end{scope}
    \begin{scope}[yshift=-3.5cm]
        \begin{scope}[xshift=-2.5cm]
            \node[vertex,label={[labelsty,label distance=-2pt]90:$c$}] (c) at (0,0) {};
            \node[vertex,label={[labelsty,label distance=-2pt]-90:$b$}] (b) at (0,-1) {};
            \node[vertex,label={[labelsty,label distance=-4pt]150:$d$}] (tl) at (150:1) {};
            \node[vertex,label={[labelsty,label distance=-4pt]30:$e$}] (tr) at (30:1) {};
            \draw[edge] (c)--(b) (c)--(tl) (c)--(tr) (tr)--(tl) (tl)--(b);
            \draw[edge,colR] (b)--(tr);
        \end{scope}
        \begin{scope}[xshift=-0cm]
            \node[vertex,label={[labelsty,label distance=-2pt]90:$a$}] (c) at (0,0) {};
            \node[vertex,label={[labelsty,label distance=-2pt]-90:$b$}] (b) at (0,-1) {};
            \node[vertex,label={[labelsty,label distance=-4pt]150:$d$}] (tl) at (150:1) {};
            \node[vertex,label={[labelsty,label distance=-4pt]30:$d'$}] (tr) at (30:1) {};
            \node[vertex,label={[labelsty,label distance=-3pt]90:$e$}] (e) at (210:0.35) {};
            \node[vertex,label={[labelsty,label distance=-3pt]90:$e'$}] (es) at (-30:0.35) {};
            \draw[edge] (c)--(tl) (c)--(tr) (tl)--(b) (b)--(tr) (c)--(e) (c)--(es) (e)--(tl) (e)--(b) (es)--(tr) (es)--(b);
            \draw[edge,colR] (tr)--(tl) (e)--(es);
        \end{scope}
        \begin{scope}[xshift=2.5cm]
            \node[vertex,label={[labelsty,label distance=-2pt]90:$c'$}] (c) at (0,0) {};
            \node[vertex,label={[labelsty,label distance=-2pt]-90:$b$}] (b) at (0,-1) {};
            \node[vertex,label={[labelsty,label distance=-4pt]150:$e'$}] (tl) at (150:1) {};
            \node[vertex,label={[labelsty,label distance=-4pt]30:$d'$}] (tr) at (30:1) {};
            \draw[edge] (c)--(b) (c)--(tl) (c)--(tr) (tr)--(tl) (b)--(tr);
            \draw[edge,colR] (tl)--(b);
        \end{scope}
        \foreach \x/\l in {-1.25/bed,1.25/be'd'}
        {
            \draw[{Latex[round]}-{Latex[round]}] ($(\x-0.5,-1)$)to node[labelsty,above] {$\l$} ($(\x+0.5,-1)$);
        }
    \end{scope}
\end{tikzpicture}
\caption{Triangulation and its associated block decomposition.}\label{structtree02fig}
\end{figure}

We could extend the black graph using the edge $ee'$ instead of $ab$, see \Cref{structtree02fig}. We now get a different decomposition with only three 4-blocks. Both rigid $\mathcal{R}_3$-circuits contain both of the red edges. There must be a $\mathcal{R}_3$-dependent set in the union of these two $\mathcal{R}_3$-circuits minus any one element of the intersection, so there must be a $\mathcal{R}_3$-dependent set in the union minus a red edge. However, removing the other red edge from this $\mathcal{R}_3$-rigid graph is only 2-connected, so removal of this red edge decreases the rank, and again we identified \dbgraph\ as a $\mathcal{R}_3$-circuit.
    
\section{Apex graphs}
\label{sec:apex}

Let $G$ be an apex graph. By definition there exists a vertex $v$ such that $G-v$ is planar. We prove the following result understanding $\mathcal{R}_3$-independence at both extremes for the degree of $v$. 

\begin{proposition}\label{prop:apex}
Let $G=(V,E)$ be an apex graph with apex vertex $v\in V$ and suppose that $d_G(v)=k$. If $k\in \{2,3,4,|V|-1\}$ then $G$ is $\mathcal{R}_3$-independent if and only if it is $(3,6)$-sparse.
\end{proposition}

The proposition is close to being best possible in a natural sense since \dbgraph\ has apex vertex $v$ of degree $6=|V|-2$ and one can apply topological 1-extensions to increase the degree of $v$ (in an apex graph that is $(3,6)$-sparse but not $\mathcal{R}_3$-independent). However, we leave open the case when the apex vertex has degree 5; we conjecture the proposition holds in this case too.

To prove the proposition we need a basic lemma. For a graph $G=(V,E)$ and a subset $X\subseteq V$, let $i_G(X)$ denote the number of edges in the subgraph of $G$ induced by $X$. In a $(3,6)$-sparse graph $G$, a set $X\subseteq V$ is \emph{critical} if $i_G(X)=3|X|-6$. Let $d(X,Y)$ denote the number of edges of the form $xy$ with $x\in X\setminus Y$ and $y\in Y\setminus X$. 
For two critical sets $X,Y\subset V$ we will implicitly use the following basic equality repeatedly:
\begin{equation*}
    i(X\cup Y)+i(X\cap Y)= i(X)+i(Y)+d(X,Y).
\end{equation*}

\begin{lemma}[{\cite[Lemma 3.1]{JJ05}}]\label{lem:1-reduce}
Let $G=(V,E)$ be $(3,6)$-sparse, let $v\in V$ be a vertex of degree 4 and suppose there exists $x,y\in N(v)$ such that $xy\notin E$. Then $G-v+xy$ is not $(3,6$)-sparse if and only if there exists a critical set $X$ such that $x,y\in X\subset V-v$.
\end{lemma}

\begin{proof}[Proof of \Cref{prop:apex}]
Let $k$ be the degree of $v$.
Suppose $k=|V|-1$. Since $G$ is $(3,6)$-sparse a short counting argument implies that $G-v$ is $(2,3)$-sparse and hence $G-v$ is $\mathcal{R}_2$-independent \cite{Geiringer,Laman}. Thus $G$ is $\mathcal{R}_3$-independent by \Cref{lem:cone}.
Suppose next that $k\in \{2,3\}$. Then $G-v$ is planar and hence $\mathcal{R}_3$-independent by \Cref{thm:gluck}. Therefore $G$ is $\mathcal{R}_3$-independent by \Cref{lem:01ext}. 

Next suppose $k=4$. For each $\{i,j\}\subset N(v)$ we define $X_{ij}$ to be the largest critical set in $V-v$ containing $i,j$ (if such a set exists).
Since $G$ is $(3,6)$-sparse, $N(v)$ does not induce a subgraph of $G$ isomorphic to $K_4$. Hence, there is $x,y\in N(v)$ such that $xy\notin E$. 
Consider $G-v+xy$. \Cref{lem:1-reduce} implies that $X_{xy}$ exists. 
Let $N(v)=\{x,y,z,w\}$.  $(3,6)$-sparsity implies that $N(v)\not\subset X_{xy}$ so without loss of generality we may suppose $w\notin X_{xy}$. 

Suppose first that $z\notin X_{xy}$. Since $xy\notin E$,
it follows from $G[X_{xy}]$ that $x,y$ are not on the boundary of the same facial triangle in $G[X_{xy}]$. 
Since $G-v$ is planar,
the vertices $w,z$ are contained in two faces $T_w,T_z$ (not necessarily distinct) of $G[X_{xy}]$ respectively.
Hence, without loss of generality,
$x$ is not on the boundary of $T_w$,
and so $xw \notin E$. 
If $G-v+xw$ is not $(3,6)$-sparse then \Cref{lem:1-reduce} implies that $X_{xw}$ exists. 
Since $x\in X_{xy}\cap X_{xw}$, $x$ is not on the boundary of $T_w$ and $G[X_{xw}]$ is 3-connected,
it follows that $X_{xw}$ contains the boundary of $T_w$.
Hence, $|X_{xy}\cap X_{xw}|\geq 3$ and $X_{xy}\cup X_{xw}$ is critical. However,
this contradicts the maximality of $X_{xy}$.

Suppose then that for every pair $\{i,j\} \subset N(v)$ where $X_{ij}$ exists,
then $X_{ij} \cap (N(v) \setminus \{i,j\}) = \emptyset$.
Since $X_{xy}$ is maximal,
then one of the edges $wx,wy,wz$ is not in $E$.
Suppose first that $wz\notin E$. If $G-v+wz$ is not $(3,6)$-sparse then \Cref{lem:1-reduce} implies that $X_{wz}$ exists and, as above, we may assume $x\notin X_{wz}$ and $y \in X_{wz}$. If $|X_{xy}\cap X_{wz}|\geq 3$ then $X_{xy}\cup X_{wz}$ is critical, contradicting the maximality of $X_{xy}$.
Hence, $|X_{xy}\cap X_{wz}| =2$.  Then $i(X_{xy}\cap X_{wz})=1$,
as otherwise $X_{xy} \cup X_{wz}$ contradicts the maximality of $X_{xy}$.
Since $X_{xy} \cup X_{wz}$ cannot be critical,
it follows from a simple sparsity counting argument that $i(X_{xy}\cup X_{wz})=3|X_{xy}\cup X_{wz}|-7$ and no edges can exist connecting vertices in $X_{xy} \setminus X_{wz}$ and $X_{wz} \setminus X_{xy}$.
Hence, $xw\notin E$. 
If $G-v+xw$ is not $(3,6)$-sparse then \Cref{lem:1-reduce} implies that $X_{xw}$ exists. Note that $|(X_{xy}\cup X_{wz})\cap X_{xw}|=3$ (since otherwise $X_{xw}$ is not maximal). Since $xw\notin E$ we have $i((X_{xy}\cup X_{wz})\cap X_{xw})\leq 2$. It follows that 
\begin{align*} i(X_{xy}\cup X_{wz}\cup X_{xw}) &\geq 3|X_{xy}\cup X_{wz}|-7+3|X_{xw}|-6+d(X_{xy}\cup X_{wz},X_{xw})-2\\
&= 3|X_{xy}\cup X_{wz}\cup X_{xw}|-6+d(X_{xy}\cup X_{wz},X_{xw}).
\end{align*}
Since $G$ is $(3,6)$-sparse, $d(X_{xy}\cup X_{wz},X_{xw})=0$ and we contradict the maximality of $X_{xy}$.
This concludes the case where $wz \notin E$.
The same proof applies for the cases where $wx \notin E$ and where $wy \notin E$.

Hence, there exists $a,b\in N(v)$ such that $G-v+ab$ is $(3,6)$-sparse.
Note also that either $G-v+ab$ is planar or it is edge-apex. Since $G-v+ab$ is $(3,6)$-sparse, it follows, from \Cref{thm:gluck} and \Cref{t:ear} respectively in these two cases, that $G-v+ab$ is $\mathcal{R}_3$-independent. Hence, \Cref{lem:01ext} implies that $G$ is $\mathcal{R}_3$-independent. 
\end{proof}

We also note that in \cite{GGJN}  $\mathcal{R}_3$-independence was understood for all graphs on at most 9 vertices. So the first open case for apex graphs is on 10 vertices. \Cref{tab:in-dep} illustrates the number of apex graphs their $(3,6)$-sparsity and their $\mathcal{R}_3$-independence for small $|V|$. 

\begin{table}[ht]
    \centering
    \begin{tabular}{rrrrrrr}
         \toprule
         $|V|$ & apex & $\mathcal{R}_3$-independent & $\mathcal{R}_3$-dependent \\\midrule
         6  & 12 & 7 & 5 &\\
         7  & 190 & 133 & 57 \\
         8  & 4482 & 3511 & 971 \\
         9  & 142142 & 115985 & 26157 \\
         10 & 5517578 & 4554816 & 962762 \\
         \bottomrule
    \end{tabular}
    \begin{tabular}{rrrrrrr}
         \toprule
         $|V|$ & $(3,6)$-sparse apex & $\mathcal{R}_3$-independent & $\mathcal{R}_3$-dependent \\\midrule
         6  & 7 & 7 & 0\\
         7  & 133 & 133 & 0 \\
         8  & 3512 & 3511 & 1 \\
         9  & 115999 & 115985 & 14 \\
         10 & 4555219 & 4554816 & 403 \\
         \bottomrule
    \end{tabular}
    \caption{Number of non-planar connected apex graphs with different independence properties \cite{DataApexIndep}.}
    \label{tab:in-dep}
\end{table}

In general we give a negative result, which combined with the existence of 3-connected apex graphs that are flexible $\mathcal{R}_3$-circuits (see \Cref{fig:3connR3circ} for an example), means we do not investigate further the rigidity of such graphs in the present article.

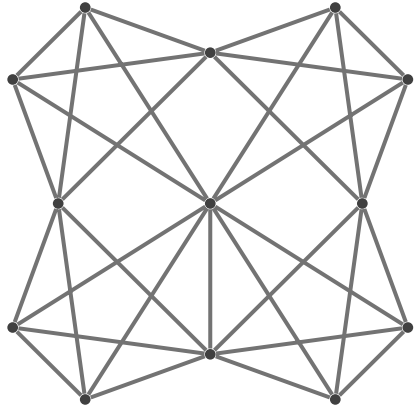
\begin{figure}[ht]
    \centering
    \begin{tikzpicture}
        \node[vertex] (o) at (0,0) {};
        \foreach \w in {0,1,2,3}
        {
            \node[vertex] (a\w) at (90*\w:2) {};
            \node[vertex] (b\w) at ($(a\w)+(90*\w+70:1.75)$) {};
            \node[vertex] (c\w) at ($(90*\w+90:2)+(90*\w+90-70:1.75)$) {};
            \draw[edge] (a\w)--(b\w) (a\w)--(c\w) (b\w)--(c\w) (b\w)--(o) (c\w)--(o);
        }
        \foreach \w [remember=\w as \wo (initially 0)] in {3,2,1,0}
        {
            \draw[edge] (a\wo)--(b\w) (a\wo)--(c\w) (a\wo)--(a\w);
        }
        \draw[edge] (a3)--(o);
    \end{tikzpicture}
    \caption{A three connected $\mathcal{R}_3$-circuit obtained from a ring of 4 copies of $K_4$ by adding a vertex of degree 9.}
    \label{fig:3connR3circ}
\end{figure}

\begin{proposition}\label{prop:apex2}
    For every nonnegative integer $g$,
    there exists a minimally $\mathcal{R}_3$--rigid apex graph with Euler genus $g$.
\end{proposition} 

\begin{proof}
Take the complete graph $K_5$, delete an edge $xy$. This is minimally $\mathcal{R}_3$-rigid of Euler genus 0. Next add a new vertex $w$ adjacent to exactly 3 vertices of the $K_5$ including both $x$ and $y$ to form a graph $G$. It is trivial to check that $G$ is minimally $\mathcal{R}_3$-rigid and the Euler genus of $G$ is 1.

Put $G=G_0$ and distinguish an edge $e=uv$. Define $G_i$, for $1\leq i \leq k$ recursively so that $G_{i+1}$ is formed from $G_i$ by gluing a copy of $G_0$ onto $G_i$ along the edge $e=uv$. (Since $G_0$ is 3-connected, in this manner we create a graph $G_k$ so that $\{u,v\}$ is a 2-vertex separating set and $G_k-\{u,v\}$ has exactly $k+1$ components each isomorphic to $G_0-\{u,v\}$.)

Since every vertex of $G_0$ is an apex, $G_0-u$ is planar and hence $G_k-u$ is planar (as the union of $k+1$ planar graphs glued together at the vertex $v$. Since each copy of $G_0$ has Euler genus 1, \Cref{addgenus} implies that $G_k$ has Euler genus $k$. 

It remains to augment $G_k$ into a minimally 
$\mathcal{R}_3$-rigid graph. We do this recursively. Since $G_0$ is minimally $\mathcal{R}_3$-rigid, and $G_1$ is obtained by gluing two copies of $G_0$ along an edge, $G_1$ is $\mathcal{R}_3$-independent by \Cref{lem:intbridge}\ref{it:intbridge:indep}  but has precisely one infinitesimal motion. We add one edge between the two copies of $G_0$ carefully (using a vertex on the outer face in each copy of $G_0-u$) and apply \Cref{lem:intbridge}\ref{it:intbridge:rank} to see that the resulting graph, which we denote $G_1^+$, is minimally $\mathcal{R}_3$-rigid. We continue in precisely this manner gluing $G_0$ onto $G_1^+$ across $uv$ and then adding an edge to form $G_2^+$, repeating until we obtain $G_k^+$. By construction each $G_i^+$ is a minimally $\mathcal{R}_3$-rigid graph with apex $u$ and Euler genus $i$.
\end{proof}

\begin{figure}[htb]
\centering
\begin{tikzpicture}[yscale=1.1,baseline={(0,0)}]
    \node[vertex,label={[labelsty]180:$a$}] (a) at (0,0) {};
    \node[vertex,label={[labelsty]180:$b$}] (b) at (60:0.4) {};
    \node[vertex,label={[labelsty]180:$c$}] (c) at (60:0.8) {};
    \node[vertex,label={[labelsty]180:$d$}] (d) at (60:1.2) {};
    \node[vertex,label={[labelsty]180:$e$}] (e) at (60:1.6) {};
    \node[vertex,label={[labelsty]0:$a$}] (a2) at ($(e)+(2,0)$) {};
    \node[vertex,label={[labelsty]0:$b$}] (b2) at ($(a2)+(-60:0.4)$) {};
    \node[vertex,label={[labelsty]0:$c$}] (c2) at ($(a2)+(-60:0.8)$) {};
    \node[vertex,label={[labelsty]0:$d$}] (d2) at ($(a2)+(-60:1.2)$) {};
    \node[vertex,label={[labelsty]0:$e$}] (e2) at ($(a2)+(-60:1.6)$) {};
    \node[vertex,label={[labelsty,label distance=-2pt]-110:$w$}] (w) at ($(e)!0.5!(a2)+(0,-0.8)$) {};
    \draw[edge] (a)--(b) (b)--(c) (c)--(d) (d)--(e);
    \draw[edge] (a2)--(b2) (b2)--(c2) (c2)--(d2) (d2)--(e2);
    \draw[edge] (a)--(e2) (a)--(d2) (e)--(w) (e)--(a2) (w)--(b2) (w)--(d2);
    \draw[edge] (a)to[bend right=40] (c);
    \draw[edge] (c)to[bend right=40] (e);
    \draw[edge] (b2)to[bend right=40] (d2);
\end{tikzpicture}
\quad
\begin{tikzpicture}[yscale=1.1,baseline={(0,0)}]
    \node[vertex,label={[labelsty]180:$a$}] (a) at (0,0) {};
    \node[vertex,label={[labelsty]180:$b'$}] (b) at (60:0.4) {};
    \node[vertex,label={[labelsty]180:$c'$}] (c) at (60:0.8) {};
    \node[vertex,label={[labelsty]180:$d'$}] (d) at (60:1.2) {};
    \node[vertex,label={[labelsty]180:$e$}] (e) at (60:1.6) {};
    \node[vertex,label={[labelsty]0:$a$}] (a2) at ($(e)+(2,0)$) {};
    \node[vertex,label={[labelsty]0:$b$}] (b2) at ($(a2)+(-60:0.4)$) {};
    \node[vertex,label={[labelsty]0:$c$}] (c2) at ($(a2)+(-60:0.8)$) {};
    \node[vertex,label={[labelsty]0:$d$}] (d2) at ($(a2)+(-60:1.2)$) {};
    \node[vertex,label={[labelsty]0:$e$}] (e2) at ($(a2)+(-60:1.6)$) {};
    \node[vertex,label={[labelsty,label distance=-2pt]-110:$w$}] (w) at ($(e)!0.5!(a2)+(0,-0.8)$) {};
    \node[vertex,label={[labelsty]180:$b$}] (br) at (-60:0.4) {};
    \node[vertex,label={[labelsty]180:$c$}] (cr) at (-60:0.8) {};
    \node[vertex,label={[labelsty]180:$d$}] (dr) at (-60:1.2) {};
    \node[vertex,label={[labelsty]180:$e$}] (er) at (-60:1.6) {};
    \node[vertex,label={[labelsty]0:$a$}] (ar2) at ($(er)+(2,0)$) {};
    \node[vertex,label={[labelsty]0:$b'$}] (br2) at ($(ar2)+(60:0.4)$) {};
    \node[vertex,label={[labelsty]0:$c'$}] (cr2) at ($(ar2)+(60:0.8)$) {};
    \node[vertex,label={[labelsty]0:$d'$}] (dr2) at ($(ar2)+(60:1.2)$) {};
    \node[vertex,label={[labelsty,label distance=-2pt]110:$w'$}] (wr) at ($(er)!0.5!(ar2)+(0,0.8)$) {};
    
    \draw[edge] (a)--(b) (b)--(c) (c)--(d) (d)--(e);
    \draw[edge] (a2)--(b2) (b2)--(c2) (c2)--(d2) (d2)--(e2);
    \draw[edge] (a)--(d2) (e)--(w) (e)--(a2) (w)--(b2) (w)--(d2);
    \draw[edge] (a)to[bend right=40] (c);
    \draw[edge] (c)to[bend right=40] (e);
    \draw[edge] (b2)to[bend right=40] (d2);
    \draw[edge] (a)--(br) (br)--(cr) (cr)--(dr) (dr)--(er);
    \draw[edge] (ar2)--(br2) (br2)--(cr2) (cr2)--(dr2) (dr2)--(e2);
    \draw[edge] (a)--(dr2) (er)--(wr) (er)--(ar2) (wr)--(br2) (wr)--(dr2);
    \draw[edge] (a)to[bend left=40] (cr);
    \draw[edge] (cr)to[bend left=40] (er);
    \draw[edge] (br2)to[bend left=40] (dr2);
\end{tikzpicture}
\quad
\begin{tikzpicture}[yscale=1.1,baseline={(0,0)}]
    \coordinate (e) at (60:1.6) {};
    \coordinate (a2) at ($(e)+(2,0)$) {};
    \node[vertex,label={[labelsty]0:$b$}] (b2) at ($(a2)+(-60:0.4)$) {};
    \node[vertex,label={[labelsty]0:$c$}] (c2) at ($(a2)+(-60:0.8)$) {};
    \node[vertex,label={[labelsty]0:$d$}] (d2) at ($(a2)+(-60:1.2)$) {};
    \node[vertex,label={[labelsty]0:$e$}] (e2) at ($(a2)+(-60:1.6)$) {};
    \node[vertex,label={[labelsty,label distance=-2pt]-110:$w$}] (w) at ($(e)!0.5!(a2)+(0,-0.8)$) {};
    \coordinate (er) at (-60:1.6) {};
    \coordinate (ar2) at ($(er)+(2,0)$) {};
    \node[vertex,label={[labelsty]0:$b'$}] (br2) at ($(ar2)+(60:0.4)$) {};
    \node[vertex,label={[labelsty]0:$c'$}] (cr2) at ($(ar2)+(60:0.8)$) {};
    \node[vertex,label={[labelsty]0:$d'$}] (dr2) at ($(ar2)+(60:1.2)$) {};
    \node[vertex,label={[labelsty,label distance=-2pt]110:$w'$}] (wr) at ($(er)!0.5!(ar2)+(0,0.8)$) {};

    \draw[edge] (b2)--(c2) (c2)--(d2) (d2)--(e2);
    \draw[edge] (w)--(b2) (w)--(d2);
    \draw[edge] (b2)to[bend right=40] (d2);
    \draw[edge] (c2)to[bend left=40] (e2);
    \draw[edge] (br2)--(cr2) (cr2)--(dr2) (dr2)--(e2);
    \draw[edge] (wr)--(br2) (wr)--(dr2);
    \draw[edge] (br2)to[bend left=40] (dr2);
    \draw[edge] (cr2)to[bend right=40] (e2);
    \draw[edge] (w)--(e2) (wr)--(e2);
\end{tikzpicture}

\caption{$G_0$, left, on the projective plane;  $G_1 =G_0 \overline{(a,e)} G_0$, middle;  $G_1\setminus a$, right, in the plane. \label{oriented13fig}}
\end{figure}
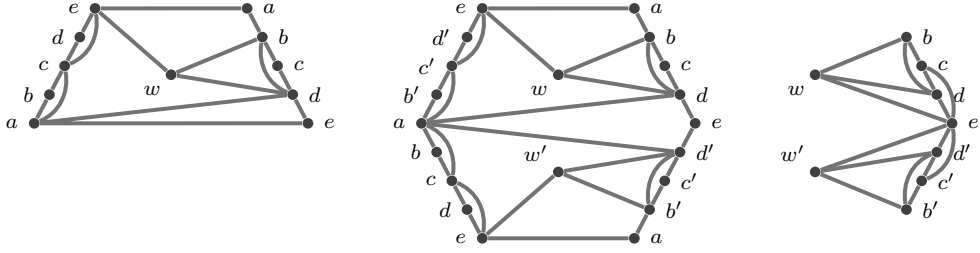

\Cref{oriented13fig} shows how to start with $G_0$
to get $G_1$ as the edge join of two $G_0$s. $G_1$ becomes planar with the apex vertex removed, and it has a cut-vertex which is the other endpoint of the edge on which the edge join was performed. It should be clear from the figure that either two $G_1$s can be "stacked" to remain orientable, in which case the Euler genus goes up by~2,  or another $G_0$ can be glued onto $G_1$, in which case the Euler genus increases by~1.

\section{Critical graphs}
\label{sec:critical}

We next analyse the rigidity of critically edge-apex graphs and critically apex graphs.

\subsection{Critically edge-apex graphs}

\begin{example}\label{ex:db}
The flexible $\mathcal{R_3}$-circuit \dbgraph\ (see \Cref{fig:doublebanana}) is non-planar. It is an apex graph and while not edge-apex it is 2-edge-apex. Moreover it is critically 8-edge apex,  since we may remove 7 edges from one of the bananas which leaves a path from top to bottom; this is still non-planar (and if we remove any 8 edges we obtain a planar graph). Likewise, it is critically 3-apex, since it is possible to remove two vertices from one banana and what is left is still non-planar (but removing any 3 vertices results in a planar graph).
\end{example}

The following basic corollary to Kuratowski's theorem~\cite{Kuratowski1930} characterises  non-planar critically edge-apex graphs.

\begin{corollary}\label{cor:kuratowksi}
    Let $G$ be a non-planar critically edge-apex graph.
    Then $G$ is a subdivision of either $K_5$ or $K_{3,3}$.
\end{corollary}

\begin{lemma}\label{lem:critedgeind}
    There is no non-planar minimally $\mathcal{R}_3$-rigid critically edge-apex graph. On the other hand, except $K_5$, every critically edge-apex graph is $\mathcal{R}_3$-independent.
\end{lemma}

\begin{proof}
It is easy to deduce from \Cref{cor:kuratowksi} that no such graph is $(3,6)$-tight and hence the first statement follows from \Cref{lem:max}.
The second conclusion follows immediately from \Cref{t:ear} since \Cref{cor:kuratowksi} implies $(3,6)$-sparsity.
\end{proof}

\begin{theorem}\label{thm:critkedgeapex}
Let $G$ be a critically $k$-edge-apex graph on $n$ vertices for any $n\geq  k+5$. Then $G$ is $\mathcal{R}_3$-independent if and only if it is $(3,6)$-sparse. Furthermore if $k\leq 7$ then $G$ is $\mathcal{R}_3$-independent if and only if it is $(3,6)$-sparse. 
\end{theorem}

\begin{proof}
The necessity is \Cref{lem:max}.
Suppose $G$ is critically k-edge-apex and not critically $(k-1)$-edge-apex. Then, for any set of edges $F=\{e_1,e_2,\dots,e_{k-1}\} \subset E$, $G-F$ is critically edge-apex and non-planar.
By \Cref{cor:kuratowksi}, $G-F$ can be constructed from $K_5$ or $K_{3,3}$ by a sequence of subdivisions.
Hence,
either $G - F$ is a subdivision of $K_5$ and hence has degree sequence $(2^{(n-5)},4^{(5)})$,
or $G - F$ is a subdivision of $K_{3,3}$ and hence has degree sequence $(2^{(n-6)},3^{(6)})$.
In either case,
$F$ would need to have size at least $n-5$ so that every degree 2 or 3 vertex in $G-F$ has degree at least 4 in $G$.
Since $|F| = k-1 < n - 5$, the minimum degree in $G$ is at most 3.

Let $v$ have degree 3. Then $G-v$ is $(3,6)$-sparse and critically $j$-edge-apex for some $j<k$. Since $n-1\geq j+5$ we can apply induction on $k$ to see that $G-v$ is $\mathcal{R}_3$-independent. (The basis of induction is \Cref{lem:critedgeind}.) That $G$ is $\mathcal{R}_3$-independent now follows from \Cref{lem:01ext}.

 Take $k\leq 7$. The only $(3,6)$-sparse graphs on at most 9 vertices that are not $\mathcal{R}_3$-independent (see \Cref{ex:db} for the unique such graph when $|V|\leq 8$ and \cite[Figure 2 and Theorem 1]{GGJN} for the three non-isomorphic graphs when $|V|=9$) are not critically $k$-edge-apex.
There exist 45 $(3,6)$-sparse $\mathcal{R}_3$-circuits with 10 vertices and 1133 $(3,6)$-sparse $\mathcal{R}_3$-circuits with 11 vertices \cite{DataCircuits}.
None of these graphs are critically $k$-edge-apex for $k\leq 7$. 
Any $\mathcal{R}_3$-dependent $(3,6)$-sparse graph on at most 11 vertices would need to contain one of these graphs,
which would violate the critical $k$-edge-apex property. The second conclusion of the theorem now follows. 
\end{proof}

Since \dbgraph\ is critically 8-edge-apex the theorem is, in a sense, best possible.

\subsection{Critically apex graphs}

To analyse critically apex graphs we use the following theorem which summarises results from 3 papers \cite{FM,F78,FL81} on the reconstruction problem. 

\begin{theorem}\label{thm:reconstruct}
Let $G$ be a graph. Then:
\begin{enumerate}
    \item\label{it:degree} if $G$ is critically apex then it has minimum degree at most 5;
    \item\label{it:reconstruct:mindeg5} if $G$ has minimum degree 5, it is critically apex if and only if it is planar;
    \item\label{it:reconstruct:mindeg4} if $G$ has minimum degree 4, it is critically apex and non-planar if and only if it is one of the graphs depicted in \Cref{fig:fiorini};
    \item\label{it:reconstruct:mindeg3} if $G$ has minimum degree 3 and either it has at least two vertices of degree 3, its order is at least 7, and $|E(G)| = 3|V (G)| - 6$, or else it has a unique vertex of degree 3 whose neighbours induce a cycle, then G cannot be a non-planar critically apex graph.
\end{enumerate}
\end{theorem}

\begin{proof}
\ref{it:degree}
Suppose not. 
Since $G-v$ is planar for all $v\in V$ it has minimum degree 5, so $G$ contains a vertex $u$ of degree $6$. Since $|E(G-u)|\leq 3|V(G-u)|-6$ we have $|E|\leq 3|V|-3$ which implies $G$ has a vertex of degree less than 6, a contradiction.  

\ref{it:reconstruct:mindeg5}-\ref{it:reconstruct:mindeg3} were proved in \cite{FM,F78,FL81}.
\end{proof}

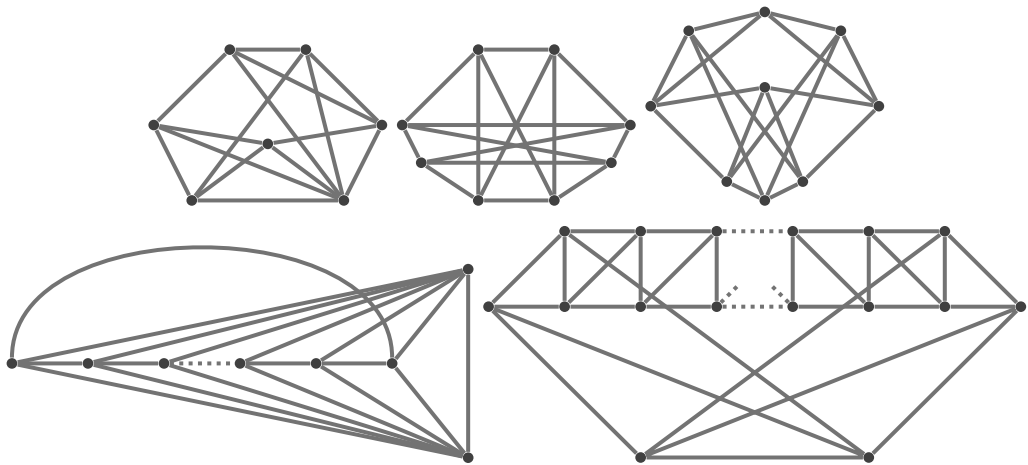
\begin{figure}[ht]
    \centering
    \begin{tikzpicture}
        \node[vertex] (1) at (0,0) {};
        \node[vertex] (2) at (2,-1) {};
        \node[vertex] (3) at (-0.5,-2) {};
        \node[vertex] (a) at (1.5,-2) {};
        \node[vertex] (b) at (-1,-1) {};
        \node[vertex] (c) at (1,0) {};
        \node[vertex] (n) at (0.5,-1.25) {};
        \draw[edge] (1)edge(c) (1)edge(2) (1)edge(b) (1)edge(a) (2)edge(c) (2)edge(n) (2)edge(a) (3)edge(a) (3)edge(b) (3)edge(c) (3)edge(n) (a)edge(n) (a)edge(c) (b)edge(n) (b)edge(a);
    \end{tikzpicture}
    \begin{tikzpicture}
        \node[vertex] (1) at (0,0) {};
        \node[vertex] (2) at (2,-1) {};
        \node[vertex] (3) at (-0,-2) {};
        \node[vertex] (a) at (1,-2) {};
        \node[vertex] (b) at (-1,-1) {};
        \node[vertex] (c) at (1,0) {};
        \node[vertex] (u) at (-0.75,-1.5) {};
        \node[vertex] (v) at (1.75,-1.5) {};
        \draw[edge] (1)edge(c) (1)edge(3) (1)edge(a) (1)edge(b) (2)edge(b) (2)edge(c) (2)edge(u) (2)edge(v) (3)edge(a) (3)edge(c) (3)edge(u) (a)edge(v) (a)edge(c) (b)edge(u) (b)edge(v) (u)edge(v);
    \end{tikzpicture}
    \begin{tikzpicture}
        \node[vertex] (1) at (-0.5,0) {};
        \node[vertex] (2) at (2,-1) {};
        \node[vertex] (3) at (-0,-2) {};
        \node[vertex] (a) at (1,-2) {};
        \node[vertex] (b) at (-1,-1) {};
        \node[vertex] (c) at (1.5,0) {};
        \node[vertex] (u) at (0.5,0.25) {};
        \node[vertex] (v) at (0.5,-0.75) {};
        \node[vertex] (w) at (0.5,-2.25) {};
        \draw[edge] (1)edge(a) (1)edge(u) (1)edge(w) (1)edge(b) (2)edge(a) (2)edge(c) (2)edge(v) (2)edge(u) (3)edge(b) (3)edge(c) (3)edge(v) (3)edge(w) (a)edge(v) (a)edge(w) (b)edge(v) (b)edge(u) (c)edge(u) (c)edge(w);
    \end{tikzpicture}

    \begin{tikzpicture}
        \node[vertex] (1) at (1,0) {};
        \node[vertex] (2) at (2,0) {};
        \node[vertex] (3) at (3,0) {};
        \node[vertex] (4) at (4,0) {};
        \node[vertex] (5) at (5,0) {};
        \node[vertex] (6) at (6,0) {};
        \node[vertex] (7) at (7,1.25) {};
        \node[vertex] (8) at (7,-1.25) {};
        \draw[edge] (1)edge(2) (2)edge(3) (4)edge(5) (5)edge(6);
        \draw[edge] (1)edge(7) (2)edge(7) (3)edge(7) (4)edge(7) (5)edge(7) (6)edge(7);
        \draw[edge] (1)edge(8) (2)edge(8) (3)edge(8) (4)edge(8) (5)edge(8) (6)edge(8);
        \draw[edge] (7)edge(8);
        \draw[edge,dotted] (3)edge(4);
        \draw[edge] (6)to[out=90,in=90](1);
    \end{tikzpicture}
    \begin{tikzpicture}
        \node[vertex] (1) at (1,0) {};
        \node[vertex] (2) at (2,0) {};
        \node[vertex] (3) at (3,0) {};
        \node[vertex] (4) at (4,0) {};
        \node[vertex] (5) at (5,0) {};
        \node[vertex] (6) at (6,0) {};
        \node[vertex] (7) at (0,-1) {};
        \node[vertex] (8) at (1,-1) {};
        \node[vertex] (9) at (2,-1) {};
        \node[vertex] (10) at (3,-1) {};
        \node[vertex] (11) at (4,-1) {};
        \node[vertex] (12) at (5,-1) {};
        \node[vertex] (13) at (6,-1) {};
        \node[vertex] (14) at (7,-1) {};
        \node[vertex] (15) at (2,-3) {};
        \node[vertex] (16) at (5,-3) {};
        \draw[edge] (1)--(2) (2)--(3) (4)--(5) (5)--(6);
        \draw[edge] (7)--(8) (8)--(9) (9)--(10) (11)--(12) (12)--(13) (13)--(14);
        \draw[edge,dotted] (3)edge(4) (10)--(11) (10)--($(10)!0.3!(4)$) (11)--($(11)!0.3!(3)$);
        \draw[edge] (7)--(1) (8)--(2) (9)--(3) (12)--(4) (13)--(5) (14)--(6);
        \draw[edge] (1)--(8) (2)--(9) (3)--(10) (4)--(11) (5)--(12) (6)--(13);
        \draw[edge] (15)--(7) (15)--(6) (15)--(14) (15)--(16) (16)--(1) (16)--(7) (16)--(14);
    \end{tikzpicture}
    \caption{Graphs found in \cite[Figure~1]{FM}. For clarity we note that the infinite family in the bottom left is the cone of a wheel and the final infinite family is 4-regular.}
    \label{fig:fiorini}
\end{figure}

This result allows us to characterise $\mathcal{R}_3$-rigidity for critically apex graphs.

\begin{theorem}\label{thm:critapex}
Let $G$ be a critically apex graph. Then $G$ is $\mathcal{R}_3$-independent if and only if $G$ is $(3,6)$-sparse.
\end{theorem}

\begin{proof}
Necessity is \Cref{lem:max}.
We prove the sufficiency. By \Cref{thm:gluck} we may assume $G$ is non-planar. Assume $G$ is critically apex and $(3,6)$-sparse. Suppose $G$ has a vertex of degree at most 3. Then $G-v$ is planar and $(3,6)$-sparse. Hence, $G-v$ is $\mathcal{R}_3$-independent by \Cref{thm:gluck} and thus $G$ is $\mathcal{R}_3$-independent by \Cref{lem:01ext}.
Then, $(3,6)$-sparsity implies $G$ has minimum degree at most 5 and, by \Cref{thm:reconstruct}, we may suppose that $G$ has minimum degree precisely 4 and is one of the graphs in \Cref{fig:fiorini}. The first 3 graphs are $(3,6)$-sparse and easily seen to be $\mathcal{R}_3$-independent. All graphs in the first infinite family are not $(3,6)$-sparse. Using \Cref{lem:01ext} it is easy to show that all graphs in the second infinite family are $\mathcal{R}_3$-independent completing the proof.
\end{proof}

To extend this to critically apex we will use a new lemma. Extending an operation used in \cite{JJT14} we say that a graph $G$ is formed from another graph $G'$ by a \emph{double-1-extension} if we obtain $G$ by deleting an edge $uv$ and adding two new vertices $a,b$ and 7 new edges $ab$, 3 edges incident to $a$ and 3 edges incident to $b$ such that $u,v \in N(a)\cup N(b)$.

\begin{lemma}\label{lem:double1ext}
Let $G'$ be $\mathcal{R}_3$-independent and let $G$ be obtained from $G'$ by a double-1-extension. Then $G$ is $\mathcal{R}_3$-independent.   
\end{lemma}

\begin{proof}
Let $G$ be obtained from $G'$ by deleting an edge $uv$ and adding two new vertices $a,b$ and 7 new edges $ab$, 3 edges incident to $a$ and 3 edges incident to $b$ such that $u,v \in N(a)\cup N(b)$.  Suppose first that $u,v\in N(a)$. Then $G$ can be obtained by a 0-extension adding $b$ followed by a 1-extension on $uv$ adding $a$. It follows from \Cref{lem:01ext} that $G$ is $\mathcal{R}_3$-independent.

Hence, we may suppose $u\in N(a)\setminus N(b)$ and $v\in N(b)\setminus N(a)$. Choose a generic framework $(G',p')$ in $\mathbb{R}^3$ and define $p$ by putting $p(x)=p'(x)$ for all $x\in V(G')$ and putting $p(a),p(b)$ on distinct points of the line defined by $p(u)$ and $p(v)$. Then $G+uv-ab$ is obtained from $G'$ by two 0-extensions so $\rank (G+uv-ab,p)=\rank (G',p')+6$ by \Cref{lem:01ext}\footnote{Technically we are using a geometric version of 0-extension, and not a generic one, but the conclusion is easy to see since $a$ and its neighbours (resp. $b$ and its neighbours) are affinely spanning.}. In $G+uv$, the 4-cycle induced by $u,a,b,v$ is realised as a collinear cycle in $(G+uv,p)$ and hence it is easy to see that 
\begin{equation*}
    \rank (G',p')+6=\rank (G+uv-ab,p)=\rank (G+uv,p)=\rank (G,p).
\end{equation*}
It follows that $G$ is $\mathcal{R}_3$-independent.
\end{proof}

\begin{theorem}\label{thm:crit2apex}
Let $G=(V,E)$ be a critically 2-apex graph. Then $G$ is $\mathcal{R}_3$-independent if and only if $G$ is $(3,6)$-sparse.
\end{theorem}

\begin{proof}
Let $G$ be critically 2-apex. Necessity is \Cref{lem:max}.
We prove sufficiency. Since $G$ is $(3,6)$-sparse the minimum degree is at most 5. By definition, $G-v$ is critically apex for all $v\in V$. 
Suppose that $v$ has degree at most 3 in $G$. Then $G-v$ is $(3,6)$-sparse and hence $\mathcal{R}_3$-independent by \Cref{thm:critapex} and \Cref{lem:01ext} implies that $G$ is $\mathcal{R}_3$-independent.
Hence, we may assume $G$ has minimum degree 4 or 5. 

If $G-v$ has minimum degree $5$ then $G-v$ is planar by \Cref{thm:reconstruct}\ref{it:reconstruct:mindeg5}. Hence, if $G-v$ has minimum degree 5 for all $v\in V$, then $G$ is critically apex and thus \Cref{thm:critapex} gives the result.

It follows that, for some $v\in V$, $G-v$ has minimum degree 3 or 4. 
Suppose $v$ is chosen so that $G-v$ has minimum degree 4. Then \Cref{thm:reconstruct}\ref{it:reconstruct:mindeg4} implies that $G-v$ is one of the graphs depicted in \Cref{fig:fiorini}. 
Since the first 3 such graphs have at most 9 vertices, $|V|\leq 10$ and hence it is easy to confirm by the data sets in \cite{DataCircuits} that $G$ is $\mathcal{R}_3$-independent. Moreover, the first infinite family contains no $(3,6)$-sparse graph, so we may assume that $G-v$ is an element of the second infinite family. 
All graphs in this family are 4-regular so $G$ has one vertex, $v$, of degree $k$ for some $k\geq 4$, $k$ vertices of degree 5 and $|V|-k-1$ vertices of degree 4. If $k=|V|-1$ then $v$ is a cone vertex and the result follows from \Cref{lem:cone} and \cite{Geiringer}. Hence, we may assume there exists a vertex $w$ of degree 4. Now by repeated application of 0- and 1-reduction operations we can reduce $G$ to $K_4$ and reverse this sequence using \Cref{lem:01ext} to show that $G$ is $\mathcal{R}_3$-independent. 

Hence, for all $u\in V$, $G-u$ has minimum degree 3. If $G$ had a unique vertex $w$ of degree 4 then $G-w$ would not have minimum degree 3. Hence, $G$ has at least two vertices of degree 4 and in fact since $G$ has minimum degree 4 and $G-u$ has minimum degree 3, $G$ has two adjacent vertices of degree 4.

Let $a,b$ denote adjacent degree 4 vertices.
If there exists distinct non-adjacent $u,v \in N_G(a) \cup N_G(b)$ such that $G' = G-\{a,b\} +uv$ is $(3,6)$-sparse,
then, as $G'$ is edge apex, \Cref{t:ear} implies that $G'$ is $\mathcal{R}_3$-independent.
As $G$ is formed from $G'$ by a double-1-extension, $G'$ is $\mathcal{R}_3$-independent by \Cref{lem:double1ext}.
Suppose this is not the case.

Suppose there exists non-adjacent $u,v \in N(a)$.
Then there exists a vertex set $X$ contained in $G-\{a,b\}$ where $u,v \in X$ and $i_G(X) = 3|X| - 6$.
Then $G[X]$ is a triangulation.
Suppose that there exists a vertex $w \in N(a) \cup N(b) \setminus X$.
Then the graph $G-\{b,w\}$ is planar and contains $G[X]$.
As $a$ is adjacent to both $u,v$ in the planar graph $G-\{b,w\}$,
then $G[X] + a + \{au,av\}$ is also planar.
This implies $G[X] + uv$ is also planar (as we can contract $au$ to get this graph).
However,
$G[X] + uv$ has $3|X|-5$ edges,
contradicting planarity.
Hence,
$N(a) \cup N(b) \subset X$.
However, this now implies $G[X + \{a,b\}]$ is not $(3,6)$-sparse,
a contradiction.

Suppose then that $N(a)$ and $N(b)$ induce copies of $K_3$.
If $N(a)=N(b)$ then $G$ contains a copy of $K_5$ contrary to $(3,6)$-sparsity. Hence, there exists $u\in N(a)\setminus N(b)$ and $v\in N(b)\setminus N(a)$. Suppose $uv\notin E$. Then there exists a vertex set $X$ contained in $G-\{a,b\}$ where $u,v \in X$ and $i_G(X) = 3|X| - 6$.
Then $G[X]$ is a triangulation.
If $N(a) \cup N(b) \subset X$ then $G[X + \{a,b\}]$ violates $(3,6)$-sparsity,
a contradicting.
If there exists two vertices $x,y \in N(a) \cup N(b) \setminus X$,
then $G[X + \{a,b\}]$ is planar.
By deleting all edges adjacent to $a,b$ except $ab,au,bv$ and then contracting $au ,bv$ we obtain a planar graph $G[X] + uv$.
However, this graph has $3|X|-5$ edges and so is not planar.

Hence,
there exists exactly one vertex $t \in N(a) \cup N(b) \setminus X$.
Let $H = G[X \cup \{a,b,t\}]$.
This subgraph is $(3,6)$-tight,
not planar (as the path $(u,a,b,v)$ could be contracted to give a non-planar subgraph), but is not 2-apex (as deleting $a$ creates a planar graph).
As $H$ is an induced subgraph,
it now must be critically apex,
and hence $V\setminus X$ contains exactly one vertex $s$.
Since $G$ is $(3,6)$-sparse and $H$ is $(3,6)$-tight,
$s$ has degree at most 3 in $G$,
contradicting our earlier assumption.
\end{proof}

Recall, from \Cref{ex:db}, that \dbgraph\ is critically 3-apex and not $\mathcal{R}_3$-independent. Hence, the theorem is tight in this sense.

\section{Global rigidity}
\label{sec:global}

A framework $(G,p)$ in $\mathbb{R}^d$ is \emph{globally $\mathcal{R}_d$-rigid} if every framework $(G,q)$ with the same edge lengths as $(G,p)$ can be obtained from $(G,p)$ by a composition of isometries. We need the following result. 

\begin{theorem}[Hendrickson \cite{Hen92}]\label{thm:hend}
Suppose $G$ is globally $\mathcal{R}_3$-rigid. Then $G-e$ is $\mathcal{R}_3$-rigid for all $e\in E$. 
\end{theorem}

We begin this section by characterising global $\mathcal{R}_3$-rigidity for edge-apex graphs. We require the following theorem of Jord\'an and Tanigawa \cite{JT19}.

\begin{theorem}\label{thm:JT}
Let $G$ be a 4-connected graph obtained from a planar triangulation by adding some number of edges. Then $G$ is globally $\mathcal{R}_3$-rigid.    
\end{theorem}

\begin{theorem}\label{t:eagr}
Let $G=(V,E)$ be an edge-apex graph.
Then the following are equivalent:
\begin{enumerate}
    \item $G$ is globally $\mathcal{R}_3$-rigid;
    \item $G$ is a rigid $\mathcal{R}_3$-circuit; and
    \item $|E|=3|V|-5$ and $G$ is 4-connected. \label{it:eagr:4con}
\end{enumerate}
\end{theorem}

\begin{proof}
By \Cref{thm:hend} if $G$ is globally $\mathcal{R}_3$-rigid then every edge is in a $\mathcal{R}_3$-circuit. Since edge-apex graphs have at most $3|V|-5$ edges, it follows from \Cref{t:ear} that $G$ must be a rigid $\mathcal{R}_3$-circuit.

Suppose $G$ is a rigid $\mathcal{R}_3$-circuit. Then $|E|=3|V|-5$. Suppose $G$ is not 4-connected, and consider a spanning triangulation $G-e$. Any 3-vertex-separation $S$ in $G$ induces a non-facial 3-cycle $f$ in $G-e$, so by \Cref{structure},  $G-e=T_1 \binom{\Delta}{f} T_2$, where $T_1, T_2$ are planar triangulations with common face $f$. Hence, both $T_1, T_2$ are $\mathcal{R}_3$-independent by \Cref{thm:gluck}. Since $G$ is not 4-connected, we may suppose that $T_1$ does not contain an end-vertex of $e$. 
Every edge in $T_1-S$ is $\mathcal{R}_3$-independent contradicting the assumption that $G$ is a $\mathcal{R}_3$-circuit.

If \ref{it:eagr:4con} holds then $G$ is obtained from a planar triangulation by adding an edge and hence the theorem follows from \Cref{thm:JT}.
\end{proof}

\begin{conjecture}
Let $G=(V,E)$ be 2-edge-apex. Then $G$ is globally $\mathcal{R}_3$-rigid if and only if $G$ is 4-connected and $G-e$ is $\mathcal{R}_3$-rigid for all edges $e \in E$.
\end{conjecture}

There are two cases to prove, either $|E|=3|V|-4$ or $|E|=3|V|-5$. In the first case there exists $e,f$ such that $G-\{e,f\}$ is a triangulation in which case the conjecture follows from Jord\'an-Tanigawa. In the second, $G-\{e,f\}$ is a triangulation minus 1 brace. Hence, \Cref{t:ear} implies that every $\mathcal{R}_3$-circuit in $G$ is $\mathcal{R}_3$-rigid. Since $|E|=3|V|-5$, $G$ contains a unique $\mathcal{R}_3$-circuit.
As $G-e$ is $\mathcal{R}_3$-rigid for every edge $e \in E$, $G$ is a $\mathcal{R}_3$-circuit. It is non-trivial though to deduce global rigidity for all such graphs.

As mentioned, it is challenging to understand $\mathcal{R}_3$-rigidity for apex graphs. Hence, it seems likely to be difficult to understand global rigidity. However, we observe the following special case.

\begin{proposition}
Let $G+v$ be an apex graph on at least 6 vertices such that $G$ is a triangulation. Then $G+v$ is globally $\mathcal{R}_3$-rigid if and only if $G+v$ is 4-connected and the subgraph of $G$ induced by the neighbour set of $v$ is not isomorphic to $K_4$.    
\end{proposition}

\begin{proof}
The necessity of 4-connectivity is clear. If $G[N(v)]=K_4$ then $v$ has degree 4. Since $G$ is a triangulation, $G+v$ contains a unique $\mathcal{R}_3$-circuit which is isomorphic to $K_5$. Since $G$ has at least 6 vertices it follows from \Cref{thm:hend} that $G+v$ is not globally $\mathcal{R}_3$-rigid, a contradiction.

Conversely, by the hypotheses we may suppose that there is a non-edge among the neighbours of $v$. Since $G+v$ is 4-connected, the graph obtained from $G$ by adding all possible edges among the neighbours of $v$ is 4-connected. These two facts imply global $\mathcal{R}_3$-rigidity by \cite{JT19}. Since $G$ is a triangulation it is $\mathcal{R}_3$-rigid so we may apply \cite[Lemma 4.1]{Tsuff} to deduce the global $\mathcal{R}_3$-rigidity of $G+v$.
\end{proof}

The next result follows easily from \Cref{cor:kuratowksi}.

\begin{lemma}
The only critically edge-apex graph that is globally $\mathcal{R}_3$-rigid is $K_5$.    
\end{lemma}

\begin{lemma}\label{lem:critwheel}
Let $G$ be critically apex. Then $G$ is globally $\mathcal{R}_3$-rigid if and only if $G$ is the cone of a wheel on at least 4 vertices.
\end{lemma}

\begin{proof}
Let $G$ be critically apex and globally $\mathcal{R}_3$-rigid. \Cref{thm:reconstruct}\ref{it:degree} implies that $G$ has minimum degree at most 5.
Since $G$ is globally $\mathcal{R}_3$-rigid the minimum degree is at least 4. 
\Cref{thm:reconstruct}\ref{it:reconstruct:mindeg5} and \ref{it:reconstruct:mindeg4} now imply that the minimum degree is precisely 4 and that the only possibilities are the graphs depicted in the two infinite families of \Cref{fig:fiorini} (the first 3 examples in the figure contradict \Cref{thm:hend}). The first family is the cone of a wheel on at least 4 vertices and the second family are $(3,6)$-sparse and hence not globally $\mathcal{R}_3$-rigid by \Cref{thm:hend}. 

The converse is to prove that the cone of a wheel on at least 4 vertices is globally $\mathcal{R}_3$-rigid. This is well known, see \cite{connelly2010global,JacksonJordan}.
\end{proof}

\section{Maximum likelihood thresholds}
\label{sec:mlt}

A fundamental question in statistical inference, informally, asks: for a fixed graph $G$, how many datapoints
are needed for the maximum likelihood estimator of the associated Gaussian graphical model to exist almost surely?  
This minimum number of 
datapoints is called the \emph{maximum likelihood threshold (MLT) of $G$}, 
which we denote $\mlt(G)$. A basic result gives the precise number for complete graphs.

\begin{lemma}[\cite{buhl1993existence}]\label{lem:complete}
    $\mlt(K_n) = n$ for all positive integers $n$.
\end{lemma}

As MLT is a monotonically increasing graph property (as in it never decreases when an edge/vertex is added),
\Cref{lem:complete} gives a weak upper bound for the MLT of any graph.
Uhler \cite{uhler2012geometry} introduced a different graph parameter called the \emph{generic completion rank (GCR)}, denoted by $\gcr(G)$, and showed that it is an upper bound for the maximum likelihood threshold. 

\begin{theorem}[\cite{uhler2012geometry}]\label{thm:uhler}
For any graph $G$, $\mlt(G) \leq \gcr(G).$    
\end{theorem}

A result of Gross and Sullivant 
\cite{gross2018maximum} implies that the generic completion rank $\gcr(G) = d+1$, 
where $d$ is the smallest dimension such that $G$ is $d$-independent. 
We take this as the definition of generic completion rank in this paper.

In \cite[Theorem 1.15]{Betal} it was proved that a graph $G$ has $\mlt (G)=d+1$ if and only if $d$ is the smallest dimension in which no generic realisation of $G$ admits a PSD stress. Furthermore, \cite[Theorem 1.18]{Betal} proved that if $G$ contains a globally $\mathcal{R}_d$-rigid subgraph on at least $d+2$ vertices then $\mlt(G)$ is at least $d+2$.

We next use the understanding we have developed for edge-apex graphs to show that for any graph $G$ in this family $\mlt (G)=\gcr (G)$ and bound these parameters for $k$-edge-apex graphs for $k\leq 3$. In general equality does not hold, with the smallest counterexample being $K_{5,5}$ (as demonstrated in \cite{blekherman2019maximum}). However, the smallest $k$ such that $K_{5,5}$ is $k$-edge-apex is $k=9$. So potentially this equality could hold in significantly more generality. 
We need two lemmas.

\begin{lemma}\label{lem:edgeapexin4d}
Let $G=(V,E)$ be 2-edge-apex. Then $G$ is $\mathcal{R}_4$-independent.    
\end{lemma}

\begin{proof}
Since a planar spanning subgraph of $G$ has at most $3|V|-6$ edges, $G$ has at most $3|V|-4$ edges and hence has a vertex of degree at most 5.
This allows us to apply an elementary induction argument (on $|V|$) using \Cref{lem:01ext}. Specifically, in the induction step,  if $G$ has a vertex $v$ of degree at most 4 then $G-v$ is 2-edge-apex (it may also be edge-apex or planar) and hence $\mathcal{R}_4$-independent by induction, and then $G$ is $\mathcal{R}_4$-independent by \Cref{lem:01ext}.
So $G$ has a vertex $v$ of degree 5.
Note that $K_6$ is not 2-edge-apex so there is a pair $x,y\in N(v)$ such that $xy\notin E$.
By definition there exist $e,f\in E$ such that $G-\{e,f\}$ is planar. 
If $\{e,f\} = \{xv, yv\}$ then $G-v$ is planar and $G-v +xy$ is edge-apex.
If $|\{e,f\} \cap \{xv, yv\}|=1$ then $G-v$ is edge-apex and $G-v+xy$ is 2-edge-apex.
Suppose $\{e,f\} \cap \{xv, yv\} = \emptyset$.
Then, in any planar embedding of $G-\{e,f\}$ we can replace the path of length two from $x$ to $y$ (in $G-\{e,f\}$) by the edge $xy$ (in $G-\{e,f\}-v$) to show that $G-\{e,f\}-v+xy$ is planar and hence $G-v+xy$ is 2-edge-apex. Hence, $G-v+xy$ is $\mathcal{R}_4$-independent by induction and then $G$ is $\mathcal{R}_4$-independent by \Cref{lem:01ext}.
\end{proof}

\begin{lemma}\label{lem:edge3apex}
Let $G=(V,E)$ be 3-edge-apex. Then $G$ is $\mathcal{R}_5$-independent. Moreover, $G$ is $\mathcal{R}_4$-independent if and only if $G$ contains no subgraph isomorphic to $K_6$. 
\end{lemma}

\begin{proof}
We proceed as in the proof of \Cref{lem:edgeapexin4d} and the first conclusion is a straightforward adaptation. So, suppose $G$ contains no subgraph isomorphic to $K_6$. If, for some $v\in V$, $d(v)\leq 4$ then $G-v$ is 3-edge-apex and, by induction, \Cref{lem:01ext} gives the result. So we may suppose some $v\in V$ has degree 5. Note that, since $|E|\leq 3|V|-3$, $G$ has at least three vertices of degree 5.
Since $G$ contains no copy of $K_6$,
there exist non-adjacent neighbours $x,y$ of $v$.
As in the proof of \Cref{lem:edgeapexin4d} we may delete $v$ and add any edge $xy$ among its neighbours (that does not already exist) to obtain a smaller 3-edge-apex graph $H=G-v+xy$. If $H$ has no subgraph isomorphic to $K_6$ then, by induction and \Cref{lem:01ext}, we are done. 

So suppose that $x,y$ is contained in a subgraph $H$ of $G-v$ isomorphic to $K_6-xy$.
Since the subgraph induced by $v$,
its neighbours and $H$ is $(3,3)$-sparse,
then $v$ has at least one neighbour not contained in $H$.
Choose another vertex $v'$ of $G$ with degree 5.
By a similar argument to that given above,
every pair of distinct non-adjacent neighbours $x',y'$ of $v'$ are contained in a copy of $K_6-xy$,
and $v'$ has at least one pair of neighbours with this property.
If $v' \in \{x,y\}$ then all the neighbours of $v$ would be contained in $H$,
which contradicts an earlier observation.
If $v'$ is otherwise contained in $H_v$ then there exists a vertex $z$ that is adjacent to all other vertices in $H$ except $v'$.
However,
the graph induced by $H +z$ would now have $19 = (3\cdot 7 - 6) + 4$ edges,
contradicting that $G$ is 3-edge-apex.
Hence, $v'$ is not contained in $H$. 

Suppose $x', y'$ are contained in a different copy of $K_6-xy$, denoted $H'$. First suppose that $x',y'$ are not both contained in $H$ (and so $x,y$ are not both contained in $H'$).
Then $H \cap H'$ is complete.
If $|V(H) \cap V(H')| \leq 2$ then it is clear that $H \cup H'$ is not 3-edge-apex.
If $3 \leq |V(H) \cap V(H')| \leq 5$ then $i(V(H \cup H')> 3|V(H\cup H')|-3$,
which contradicts the hypothesis that $H \cup H'$ is 3-edge-apex.
Hence, $x',y'$ are contained in $H$.

It remains to deal with the case when $H=H'$.
From all of our previous arguments,
it follows that every vertex of degree 5 is adjacent to both $x$ and $y$.
Fix $t$ to be the number of degree 5 vertices in $G$.
Then $d_G(x),d_G(y) \geq 4 + t$.
As $G$ is 3-edge-apex, $|E| \leq (3|V| - 6) + 3$,
so
\begin{equation*}
    6|V| - 6 \geq 2|E| \geq 5t + 2(4+t) + 6(|V| - t - 2) = 6 |V|+t -4.
\end{equation*}
However, this implies $t \leq -2$,
which is clearly false.
This contradiction completes the proof.
\end{proof}

\begin{theorem}\label{thm:mltea}
    Let $G$ be a $k$-edge-apex graph with $k\leq 3$.
    Then either:
    \begin{enumerate}
        \item \label{it:thm:mltea1} $k=1$ and $\mlt(G) = \gcr(G) \leq 5$;
        \item \label{it:thm:mltea2} $k=2$ and $\mlt(G)\leq \gcr(G)\leq 5$; or
        \item \label{it:thm:mltea3} $k=3$ and either $G$ contains a subgraph isomorphic to $K_6$ and $\mlt (G)=\gcr (G) = 6$, or $G$ has no subgraph isomorphic to $K_6$ and $\mlt(G) \leq \gcr(G) \leq 5$.
    \end{enumerate}
\end{theorem}

\begin{proof}
    \ref{it:thm:mltea1}: Since there exists $e$ so that $G-e$ is planar and hence 3-independent (by \Cref{thm:gluck}), it follows that $G$ is 4-independent and therefore $\gcr(G) \leq 5$.
    If $\gcr(G) \leq 4$ then $\mlt(G) = \gcr(G)$ by \cite[Theorem 1.19]{Betal}.
    Suppose that $\gcr(G) = 5$.
    Then $G$ contains a $\mathcal{R}_3$-circuit $H$.
    By \Cref{t:eagr},
    $H$ is globally $\mathcal{R}_3$-rigid. It is simply unpacking the definition to check that $\mlt (H) \leq \mlt (G)$. Hence, by \Cref{thm:uhler}, we have
    \begin{align*}
        5 \leq \mlt(H) \leq \mlt(G) \leq \gcr(G) \leq 5,
    \end{align*}
    and so $\mlt(G) = \gcr(G)$ as required.

    \ref{it:thm:mltea2}: The first inequality is \Cref{thm:uhler} and the second is immediate from \Cref{lem:edgeapexin4d}. 

    \ref{it:thm:mltea3}: The conclusions follow from \Cref{lem:edge3apex}, the fact that $K_6$ has a generic realisation with a PSD stress in 4-dimensions and \Cref{thm:uhler}.
\end{proof}

We now consider apex graphs. The following is immediate after combining \Cref{lem:cone} and \Cref{thm:gluck}.

\begin{lemma}\label{lem:apexddim} 
Every $k$-apex graph is $\mathcal{R}_{3+k}$-independent. 
\end{lemma}

Moreover by choosing a triangulation we see that the iterated cone is an example of a $k$-apex graph that is not $\mathcal{R}_{3+k-1}$-independent. A specific such example is $K_{k+4}$. 

Note that in the case when $k=1$, \Cref{lem:apexddim}  follows from the fact that linklessly embeddable graphs are $\mathcal{R}_4$-independent \cite{Nevo}. 

Similarly starting from the graph in \Cref{fig:3connR3circ} and iterating the coning operation, using \Cref{lem:cone}, we obtain.

\begin{lemma}
For all $k\geq 1$ there exist $k$-apex graphs that are $(k+2)$-connected flexible $\mathcal{R}_{k+2}$-circuits.    
\end{lemma}

It is easy to deduce from \Cref{lem:cone} the following (which is tight, e.g. repeatedly cone a triangulation). See also \cite[Lemma 2.9]{Betal2} which shows that coning adds precisely 1 to the MLT.

\begin{proposition}\label{apexmlt}
For any $k\geq 1$, let $G$ be a $k$-apex graph. Then $\mlt \leq \gcr (G) \leq 4+k$.
\end{proposition}

It is tempting to conjecture that $\mlt (G)=\gcr (G)$ in the case when $k=1$ (i.e. $G$ is an apex graph). To prove this it suffices to show that every apex graph $G$ that is a $\mathcal{R}_3$-circuit has a 3-dimensional realisation with a PSD stress. We know this is true if $G$ is globally rigid in $\mathbb{R}^3$ by applying \cite{CGT}. It may be tractable when $G$ is $\mathcal{R}_3$-rigid. It is also tractable for flexible $\mathcal{R}_3$-circuits, such as in \Cref{fig:doublebanana}, that have a sufficiently special form to apply a gluing construction for PSD stresses \cite[Lemma 6.6]{Betal}. However, not all  flexible $\mathcal{R}_3$-circuits that are apex graphs have this form; an example is given in \Cref{fig:3connR3circ}.

We next analyse $\mlt(G)$ when $G$ is critically apex and critically edge-apex.

\begin{proposition}\label{prop:critapexmlt}
    Let $G$ be a critically apex graph. 
    Then $\mlt = \gcr (G) \leq 5$,
    with equality if and only if $G$ is the cone of a wheel on at least 4 vertices.
\end{proposition}

\begin{proof}
    If $G$ is $(3,6)$-sparse,
    then it is $\mathcal{R}_3$-independent by \Cref{thm:critapex} and hence $\gcr(G) \leq 4$. In this case the result follows from \cite[Theorem 1.19]{Betal}.

    Suppose $G$ is not $(3,6)$-sparse.
    Hence, $G$ is not planar, but every induced subgraph of $G$ is planar.
    Since all planar graphs are $\mathcal{R}_3$-independent (\Cref{thm:gluck}),
    it follows that all $\mathcal{R}_3$-circuits contained in $G$ must also span $G$.
    As noted in the proof of \Cref{lem:critwheel},
    $G$ has minimum degree at most 5.
    If $G$ has minimum degree 3 or less then it does not contain a spanning $\mathcal{R}_3$-circuit,
    and if $G$ has minimum degree 5 then it is planar.
    Hence, $G$ has minimum degree 4.
    \Cref{thm:reconstruct} now implies that $G$ is  the cone of a wheel on at least 4 vertices (since the other options are $(3,6)$-sparse). Hence, $G$ is globally $\mathcal{R}_3$-rigid. Since $G$ is $\mathcal{R}_4$-independent we now have $\mlt(G)=\gcr(G)=5$.
\end{proof}

Finally we deal with critically edge-apex graphs.

\begin{lemma}
Let $G$ be critically edge-apex.  Then $\mlt(G)=\gcr(G)$.  
\end{lemma}

\begin{proof}
By \Cref{cor:kuratowksi}, $G$ can be constructed from $K_5$ or $K_{3,3}$ by a sequence of subdivisions. Any graph obtained from $K_{3,3}$ by a sequence of subdivisions is $\mathcal{R}_2$-independent (since $K_{3,3}$ is) and hence $\gcr(G)=3$. The result follows from \cite[Theorem 1.19]{Betal}.   

So suppose that $G$ is obtained from $K_5$ by a sequence of subdivisions. If $G=K_5$ then $G$ is globally $\mathcal{R}_3$-rigid and hence (by \Cref{lem:complete}) $\mlt(G)=5$. Since $K_5$ is $\mathcal{R}_4$-independent, $\mlt(G)=\gcr(G)$. If $G$ is obtained  by subdividing at least one edge then $G$ is $\mathcal{R}_3$-independent and hence $\gcr(G)\leq 4$ so the result again follows from \cite[Theorem 1.19]{Betal}.
\end{proof}

For the reader interested in maximum likelihood thresholds convenience we tabulate the cases we have established in \Cref{tab:mlt}. Further known results can be found in \cite{Betal2,Betal}.

\begin{table}[ht]
    \centering
    \begin{tabular}{lll}
         \toprule
         Graph & GCR & MLT \\\midrule
         edge-apex & $\leq 5$ & $=\gcr$  \\
        2-edge-apex & $\leq 5$ & $=\gcr$   \\
        3-edge-apex with no $K_6$ & $\leq 5$ & $\leq \gcr$   \\
        3-edge-apex with a $K_6$ & $=6$ & $=\gcr$   \\
        $k$-apex & $\leq 4+k$ & $\leq \gcr$  \\
        critically apex & $\leq 5$ & $=\gcr$   \\
        critically edge-apex & $\leq 5$ & $=\gcr$ \\
         \bottomrule
    \end{tabular}
    \caption{Equality of maximum likelihood threshold and generic completion rank for the graph classes in this paper.}
    \label{tab:mlt}
\end{table}

\addcontentsline{toc}{section}{Acknowledgments}
\section*{Acknowledgements}

This project originated from the Fields Institute Focus Program on Geometric Constraint Systems.
The authors are grateful to the Fields Institute for their hospitality and financial support. We thank Katie Clinch, Tony Huyn and Bill Jackson for helpful discussions.
S.\,D.\ was supported by the Heilbronn Institute for Mathematical Research. G.\,G.\ was partially supported by the Austrian Science Fund (FWF): 10.55776/P31888. A.\,N.\ was partially supported by EPSRC grant EP/X036723/1.

\bibliographystyle{plainurl}
{\footnotesize
\bibliography{references}
}
\end{document}